\documentclass[11pt]{article}

\usepackage[sc]{mathpazo}

\usepackage[margin=1in]{geometry}
\usepackage[english]{babel}
\usepackage[utf8x]{inputenc}
\usepackage[compact]{titlesec}

\usepackage{cmap}
\usepackage[T1]{fontenc}
\usepackage{bm}
\usepackage{amsmath}
\usepackage{amsfonts}
\usepackage{amssymb}
\usepackage{amsbsy}
\usepackage{amsthm}

\usepackage{graphicx, ucs}

\usepackage{subcaption}
\usepackage{rotating}
\usepackage{float}
\usepackage{tikz}

\usepackage{algorithm}
\usepackage[noend]{algpseudocode}
\usepackage{listings}

\usepackage{enumitem}
\usepackage{hyperref}
\hypersetup{
  colorlinks = true,
  urlcolor = {blueGrotto},
  linkcolor = {royalBlue},
  citecolor = {limeGreen}
}

\usepackage{multirow}
\usepackage{array}
 \usepackage{chngcntr}
\usepackage{soul}
\usepackage{dsfont}

\def\compactify{\itemsep=0pt \topsep=0pt \partopsep=0pt \parsep=0pt}
\let\latexusecounter=\usecounter

\newenvironment{Enumerate}
  {\def\usecounter{\compactify\latexusecounter}
   \begin{enumerate}}
  {\end{enumerate}\let\usecounter=\latexusecounter}

\def\eps{\varepsilon}

\DeclareMathOperator*{\argmin}{argmin}

\def\abs#1{\left|#1\right|}

\newcommand{\paragr}[1]{\noindent \textbf{#1}}

\def\matr#1{\boldsymbol{#1}}
\renewcommand{\vec}[1]{\boldsymbol{#1}}

\def\norm#1{\left\|#1\right\|}

\def\poly{\mathrm{poly}}

\DeclareMathOperator*{\Exp}{{\mathbb{E}}}
\DeclareMathOperator*{\Prob}{{\mathbb{P}}}
\DeclareMathOperator*{\Cov}{{\mathrm{Cov}}}

\def\normal{\mathcal{N}}

\newcommand{\reals}{\mathbb{R}}

\newtheorem{theorem}{Theorem}

\newtheorem{proposition}{Proposition}
\newtheorem{lemma}{Lemma}
\newtheorem{claim}{Claim}
\newtheorem{corollary}{Corollary}
\newtheorem{remark}{Remark}

\newtheorem{definition}{Definition}

\newif\ifnotes\notestrue

\ifnotes
\usepackage{color}
\definecolor{mygrey}{gray}{0.50}
\newcommand{\notename}[2]{{\textcolor{red}{\footnotesize{\bf (#1:} {#2}{\bf
) }}}}

\else

\newcommand{\notename}[2]{{}}

\fi

 \definecolor{niceRed}{RGB}{190,38,38}
\definecolor{blueGrotto}{HTML}{059DC0}
\definecolor{royalBlue}{HTML}{057DCD}
\definecolor{navyBlueP}{HTML}{0B579C}
\definecolor{limeGreen}{HTML}{81B622}

\newcommand{\memb}{M}
\newcommand{\chara}{\mathds{1}}
\newcommand{\Hessian}{\mathbf{H}}

\newcommand{\diag}{\mathrm{diag}}

\newcommand{\symm}{\mathcal{Q}}

\newcommand{\samplO}{\mathcal{O}}
\newcommand{\Domain}{\mathcal{D}}

\newcommand{\dtv}{\mathrm{d}_{\mathrm{TV}}}

\begin{document}

\title{Efficient Statistics, in High Dimensions, \\ from Truncated Samples}
\author{
  \textbf{Constantinos Daskalakis} \\
  \small Massachusetts Institute of Technology \\
  \url{costis@csail.mit.edu}
  \and
  \textbf{Themis Gouleakis} \\
  \small Max Plank Institute for Informatics \\
  \url{themis.gouleakis@gmail.com}
  \and
  \textbf{Christos Tzamos} \\
  \small University of Wisconsin-Madison \\
  \url{tzamos@wisc.edu}
  \and
  \textbf{Manolis Zampetakis} \\
  \small Massachusetts Institute of Technology \\
  \url{mzampet@mit.edu}
}
\maketitle

\begin{abstract}
    We provide an efficient algorithm for the classical problem, going back to
  Galton, Pearson, and Fisher, of estimating, with arbitrary accuracy the
  parameters of a multivariate normal distribution from \textit{truncated
  samples}. Truncated samples from a $d$-variate normal ${\cal
  N}(\vec{\mu},\matr{\Sigma})$ means a samples is only revealed if it falls in
  some subset $S \subseteq \mathbb{R}^d$; otherwise the samples are hidden and
  their count in proportion to the revealed samples is also hidden. We show that
  the mean $\vec{\mu}$ and covariance matrix $\matr{\Sigma}$ can be estimated
  with arbitrary accuracy in polynomial-time, as long as we have oracle access
  to $S$, and $S$ has non-trivial measure under the unknown $d$-variate normal
  distribution. Additionally we show that without oracle access to $S$, any
  non-trivial estimation is impossible.
\end{abstract}

  \section{Introduction} \label{sec:intro}

  A classical challenge in Statistics is estimation from truncated or censored
samples. Truncation occurs when samples falling outside of some subset $S$ of
the support of the distribution are not observed, and their count in proportion
to the observed samples is also not observed. Censoring is similar except the
fraction of samples falling outside of $S$ is given. Truncation and censoring of
samples have myriad manifestations in business, economics, manufacturing,
engineering, quality control, medical and biological sciences, management
sciences, social sciences, and all areas of the physical sciences. As a simple
illustration, the values that insurance adjusters observe are usually
left-truncated, right-censored, or both. Indeed, clients usually only report
losses that are over their deductible, and may report their loss as equal to the
policy limit when their actual loss exceeds the policy limit as this is the
maximum that the insurance company would pay.

  Statistical estimation under truncated or censored samples has had a long
history in Statistics, going back to at least the work of Daniel Bernoulli who
used it to demonstrate the efficacy of smallpox vaccination in
1760~\cite{Bernoulli1766}. In 1897, Galton analyzed truncated samples
corresponding to registered speeds of American trotting
horses~\cite{Galton1897}. His samples consisted of running times of horses that
qualified for registration by trotting around a one-mile course in not more than
2 minutes and 30 seconds while harnessed to a two-wheeled cart carrying a weight
of not less than 150 pounds including the driver. No records were kept for the
slower, unsuccessful trotters, and their number thus remained unknown. Assuming
that the running times prior to truncation were normal, Galton applied simple
estimation procedures to estimate their mean and standard deviation.
Dissatisfaction with Galton's estimates led Pearson~\cite{Pearson1902} and later
Pearson and Lee~\cite{PearsonLee1908} and Lee~\cite{Lee1915} to use the method
of moments in order to estimate the mean and standard deviation of a univariate
normal distribution from truncated samples. A few years later, Fisher used the
method of maximum likelihood to estimate univariate normal distributions from
truncated samples~\cite{fisher31}.

  Following the early works of Galton, Pearson, Lee and Fisher, there has been a
large volume of research devoted to estimating the truncated normal or other
distributions; see e.g.~\cite{Schneider86,Cohen91,BalakrishnanCramer} for an
overview of this work. However, estimation methods, based on moments or maximum
likelihood estimation, are intractable for high-dimensional data and are only
known to be consistent in the limit, as the number of samples tends to infinity,
even for normal distributions. With infinitely many samples, it seems intuitive
that given the density of a multivariate normal ${\cal
N}(\vec{\mu},\matr{\Sigma})$ conditioned on a measurable set $S$, the ``local
shape'' of the density inside $S$ should provide enough information to
reconstruct the density everywhere. Indeed, results of
Hotelling~\cite{Hotelling48} and Tukey~\cite{Tukey49} prove that the conditional
mean and variance on any measurable $S$ are in one-to-one correspondence with
the un-conditional parameters. When the sample is finite, however, it is not
clear what features of the sample to exploit to estimate the parameters, and in
particular it is unclear how sensitive to error is the correspondence between
conditional and unconditional parameters. To illustrate, in
Figure~\ref{fig:figure1}, we show one thousand  samples from two bi-variate
normals, which are far in total variation distance, truncated to a box.
Distinguishing between the two Gaussians is not immediate despite the large
total variation distance between these normals.
\begin{figure}
	\centering
		\includegraphics[width=.60\textwidth]{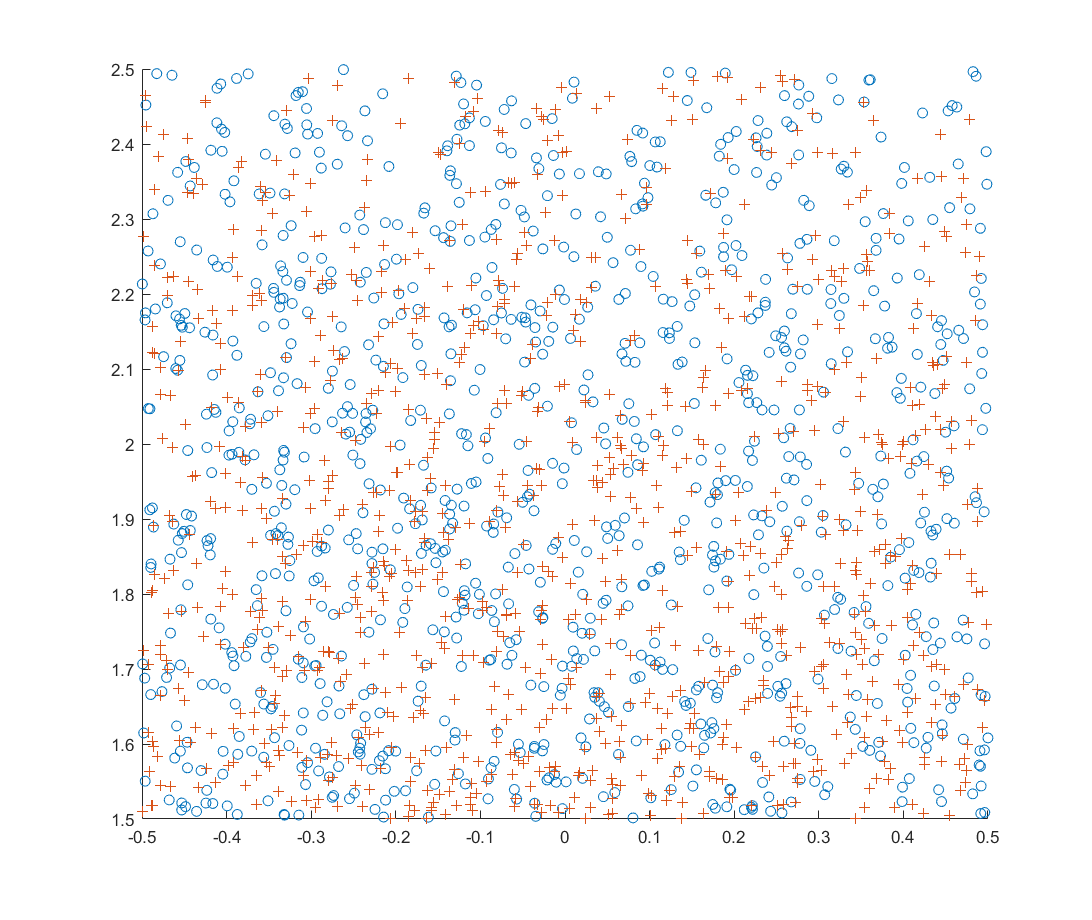}
			\caption{One thousand samples from ${\cal N}((0,1), I)$ and from
			${\cal N}((0,0), 4I)$ truncated into the $[-0.5,0.5]\times [1.5,2.5]$ box.
			We leave it to the reader to guess which sample comes from which.}
	\label{fig:figure1}
\end{figure}

  In this paper, we revisit this classical problem of multivariate normal
estimation from truncated samples to obtain polynomial time and sample
algorithms, while also accommodating a very general truncation model. We suppose
that samples, $\vec{x}^{(1)}, \vec{x}^{(2)}, \ldots$, from an unknown
$d$-variate normal  $\normal(\vec{\mu},\matr{\Sigma})$ are only revealed if they
fall into some subset $S \subseteq \mathbb{R}^d$; otherwise the samples are
hidden and their count in proportion to the revealed samples is also hidden. We
make no assumptions about $S$, except that its measure with respect to the
unknown distribution is non-trivial, say $\alpha=1\%$, and that we are given
oracle access to this set, namely, given a point $x$ the oracle outputs
$\chara\{\vec{x} \in S\}$. Otherwise, set $S$ can be any measurable set, and in
particular need not be convex. In contrast, to the best of our knowledge,  prior
work only considers sets $S$ that are boxes, while still not providing
computational tractability, or finite sample bounds for consistent estimation.
We provide the first time and sample efficient estimation algorithms even for
simple truncation sets, but we also accommodate very general sets. This, in
turn, enables statistical estimation in settings where set $S$ is determined by
a complex set of rules, as it happens in many important applications, especially
in high-dimensional settings. Revisiting our earlier example, insurance policies
on a collection of risks may be complex, so the adjustor's observation set $S$
may be determined by a complex function on the loss vector $\vec{x}$.

  Our main result is that the mean vector $\vec{\mu}$ and covariance matrix
$\matr{\Sigma}$ of an unknown $d$-variate normal can be estimated to arbitrary
accuracy in polynomial-time from a truncated sample. In particular,

\begin{theorem} \label{thm:main theorem}
	  Given oracle access to a measurable set $S$, whose measure under some
	unknown $d$-variate normal ${\cal N}(\vec{\mu},\matr{\Sigma})$ is at least
	some constant $\alpha>0$, and samples $\vec{x}^{(1)}, \vec{x}^{(2)}, \ldots$
	from $\normal(\vec{\mu},\matr{\Sigma})$ that are truncated to this set, there
	exists a polynomial-time algorithm that recovers estimates $\hat{\vec{\mu}}$
	and $\hat{\matr{\Sigma}}$. In particular, for all $\epsilon >0$, the algorithm
	uses $\tilde{O}(d^2/\eps^2)$ truncated samples and queries to the oracle and
	produces estimates that satisfy the following with probability at least
	$99\%$:
	\begin{align}
		\norm{\matr{\Sigma}^{-1/2}(\vec{\mu}-\hat{\vec{\mu}})}_{2} \le
		\eps;~\quad~\text{and}~\quad~
		\norm{\matr{I}-\matr{\Sigma}^{-1/2}\hat{\matr{\Sigma}}\matr{\Sigma}^{-1/2}}_{F} \le
		\eps.
	  \label{eq:estimation guarantees}
	\end{align} Note that under the above
	conditions the total variation distance between ${\cal
	N}(\vec{\mu},\matr{\Sigma})$ and ${\cal
	N}(\hat{\vec{\mu}},\hat{\matr{\Sigma}})$ is $O(\eps)$, and the number of
	samples used by the algorithm is optimal, even when there is no truncation,
	i.e.~when $S = \mathbb{R}^d$.
\end{theorem}

  It is important to note that the measure $\alpha$ assigned by the unknown
distribution on $S$ can be arbitrarily small, yet the accuracy of estimation can
be driven to arbitrary precision. Moreover, we note that without oracle access
to the indicator $\vec{1}_{x \in S}$, it is information-theoretically impossible
to even attain a crude approximation to the unknown normal, even in one
dimension, namely,

\begin{theorem} \label{thm:lower bound}
	  For all $\alpha > 0$, given infinitely many samples from a univariate normal
	${\cal N}(\mu, \sigma^2)$, which are truncated to an unknown set $S$ of
	measure $\alpha$, it is impossible to estimate parameters $\hat{\mu}$ and
	$\hat{\sigma}^2$ such that the distributions ${\cal N}(\mu, \sigma^2)$ and
	${\cal N}(\hat{\mu}, \hat{\sigma}^2)$ are guaranteed to be within
	$\frac{1-\alpha}{2}$ in total variation distance.
\end{theorem}

\paragraph{Overview of the Techniques.} The proofs of Theorems~\ref{thm:main
theorem} and~\ref{thm:lower bound} are provided in Sections~\ref{sec:upper
bound} and~\ref{sec:lower bound} respectively. Here we provide an overview of
our proof of Theorem~\ref{thm:main theorem}. Our algorithm shown in
Figure~\ref{fig:algorithm} is (Projected) Stochastic Gradient Descent (SGD) on
the negative log-likelihood of the truncated samples. Notice that we cannot
write a closed-form of the log-likelihood as the set $S$ is arbitrary and
unknown to us. Indeed, we only have oracle access to this set and can thus not
write down a formula for the measure of $S$ under different estimates of the
parameters. While we cannot write a closed-form expression for the negative
log-likelihood, we still show that it is convex for arbitrary sets $S$, as long
as we re-parameterize our problem in terms of
$\vec{v}=\matr{\Sigma}^{-1}\vec{\mu}$ and $\matr{T}=\matr{\Sigma}^{-1}$ (see
Lemma~\ref{lemma:reparameterized convex}). Using anti-concentration of
polynomials of the Gaussian measure, we show that the negative log-likelihood is
in fact strongly convex, with a constant that depends on the measure of $S$
under the current estimate $(\vec{v},\matr{T})$ of the parameters (see
Lemma~\ref{lemma:strong convexity}). In particular, to maintain strong
convexity, SGD must remain within a region of parameters $(\vec{v},\matr{T})$
that assign non-trivial measure to $S$. We show that the pair of parameters
$(\vec{v},\matr{T})$ corresponding to the conditional (on $S$) mean and
covariance, which can be readily estimated from the truncated sample, satisfies
this property (see Corollary~\ref{corollary:initialization}). So we use these as
our initial estimation of the parameters. Moreover, we define a convex set of
parameters $(\vec{v},\matr{T})$ that all assign non-trivial measure to $S$. This
set contains both our initialization and the ground truth (see
Lemmas~\ref{lemma:sufficient constraint} and~\ref{lemma:initialization satisfies
sufficient}), and we can also efficiently project on it (see
Lemma~\ref{lemma:efficient projection}). So we run our Projected Gradient
Descent procedure on this set. As we have already noted, we have no closed-form
for the log-likelihood or its gradient. Nevertheless, we show that, given oracle
access to set $S$, we can get un-biased samples for the gradient of the
log-likelihood function  using rejection sampling from the normal distribution
defined our current estimate of the parameters $(\vec{v},\matr{T})$ (see
Lemma~\ref{lemma:unbiased estimate of sample}). For this additional reason, it
is important to keep the invariant that SGD remains within a set of parameters
that all assign non-trivial measure to $S$.

\paragraph{Related work.} We have already discussed work on censored and
truncated statistical estimation. More broadly, this problem falls in the realm
of robust statistics, where there has been a strand of recent works studying
robust estimation and learning in high dimensions. A celebrated result by Candes
et al.~\cite{CLMW11} computes the PCA of a matrix, even when a constant
fraction of its entries to be adversarially corrupted, but it requires the
locations of the corruptions to be relatively evenly distributed. Related work
of Xu et al. \cite{XCM10} provides a robust PCA algorithm for arbitrary
corruption locations, requiring at most $50\%$ of the points to be corrupted.

Closer to our work, \cite{DKK+16b,LRV16,DKK+17,DKK+18} perform robust estimation
of the parameters of multi-variate Gaussian distributions in the presence of
corruptions to a small $\varepsilon$ fraction of the samples, allowing for both
deletions of samples and additions of samples that can also be chosen adaptively
(i.e.~after seeing the sample generated by the Gaussian). The authors in
\cite{CSV17} show that corruptions of an arbitrarily large fraction of samples
can be tolerated as well, as long as we allow ``list decoding'' of the
parameters of the Gaussian. In particular, they design learning algorithms that
work when an $(1 - \alpha)$-fraction of the samples can be adversarially
corrupted, but output a set of $poly(1/\alpha)$ answers, one of which is
guaranteed to be accurate.

Similar to~\cite{CSV17}, we study a regime where only an arbitrarily small
constant fraction of the samples from a normal distribution can be observed. In
contrast to~\cite{CSV17}, however, there is a fixed set $S$ on which the samples
are revealed without corruption, and we have oracle access to this set. The
upshot is that we can provide a single accurate estimation of the normal rather
than a list of candidate answers as in \cite{CSV17}, while accommodating
a much larger number of deletions of samples compared to \cite{DKK+16b,DKK+18}.

Other robust estimation works include robust linear regression \cite{BJK15} and
robust estimation algorithms under sparsity assumptions \cite{Li17,BDLS17}.
In~\cite{HM13}, the authors study robust subspace recovery having both upper and
lower bounds that give a trade-off between efficiency and robustness. Some
general criteria for robust estimation are formulated in \cite{SCV18}.

  \section{Preliminaries} \label{sec:model}

\paragr{Notation.} We use $\langle \vec{x}, \vec{y} \rangle$ for the inner
product of $\vec{x}, \vec{y} \in \reals^d$. We use $\matr{I}_d$ to refer to the
identity matrix in $d$ dimensions, where we may skip the subscript when the
dimensions are clear. We use $\matr{E}_{i, j}$ to refer to the all zero matrix
with one $1$ at the $(i, j)$ entry. Let $\matr{A} \in \reals^{d \times d}$, we
define $\matr{A}^{\flat} \in \reals^{d^2}$ to be the standard vectorization of
$\matr{A}$. We define $\sharp$ to be the inverse operator, i.e.
$\left(\matr{A}^{\flat}\right)^{\sharp} = \matr{A}$. Let also $\symm_d$ be the
set of all the symmetric $d \times d$ matrices. The covariance matrix between
two random variables $\vec{x}, \vec{y}$ is $\Cov[\vec{x}, \vec{y}]$.
\smallskip

\paragr{Vector and Matrix Norms.} We define the $\ell_p$-norm of
$\vec{x} \in \reals^{d}$ to be
$\norm{\vec{x}}_p = \left( \sum_i x_i^p \right)^{1/p}$ and the
$\ell_{\infty}$-norm of $\vec{x}$ to be
$\norm{\vec{x}}_{\infty} = \max_{i} \abs{x_i}$. We also define the
spectral norm of a matrix $\matr{A}$ is equal to
\[ \norm{\matr{A}}_2 = \max_{\vec{x} \in \reals^d, \vec{x}
\neq 0} \frac{\norm{\matr{A} \vec{x}}_2}{\norm{\vec{x}}_2}. \]
It is well known that $\norm{\matr{A}}_2 = \max\{\abs{\lambda_i}\}$, where
$\lambda_i$'s are the eigenvalues of $\matr{A}$. The Frobenius norm of a matrix
$\matr{A}$ is defined as $\norm{\matr{A}}_F=\norm{\matr{A}^{\flat}}_2$. The
\emph{Mahalanobis distance} between two vectors $\vec{x}$, $\vec{y}$ given a
covariance matrix $\Sigma$ is defined as
\[ \norm{\vec{x}-\vec{y}}_{\Sigma}=\sqrt{(\vec{x}-\vec{y})^T\matr{\Sigma}^{-1}(\vec{x}-\vec{y})}. \]

\paragr{Truncated Gaussian Distribution.} Let
$\normal(\vec{\mu}, \matr{\Sigma})$ be the normal distribution with mean
$\vec{\mu}$ and covariance matrix $\matr{\Sigma}$, with the following
probability density function
\begin{align} \label{eq:normalDensityFunction}
  \normal(\vec{\mu}, \matr{\Sigma}; \vec{x}) =
  \frac{1}{\sqrt{\det(2 \pi \matr{\Sigma})}}
  \exp \left( - \frac{1}{2} (\vec{x} - \vec{\mu})^T \matr{\Sigma}^{-1}
  (\vec{x} - \vec{\mu}) \right).
\end{align}
Also, let $\normal(\vec{\mu}, \matr{\Sigma}; S)$ denote the
\textit{probability mass of a set} $S \subseteq \reals^d$ under the Gaussian
measure $\normal(\vec{\mu}, \matr{\Sigma})$. Let $S \subseteq \reals^d$ be a
subset of the $d$-dimensional Euclidean space, we define the
\textit{$S$-truncated normal distribution}
$\normal(\vec{\mu}, \matr{\Sigma}, S)$ the normal distribution
$\normal(\vec{\mu}, \matr{\Sigma})$ conditioned on taking values in the subset
$S$. The probability density function of $\normal(\vec{\mu}, \matr{\Sigma}, S)$
is the following
\begin{equation} \label{eq:truncatedNormalDensityFunction}
  \normal(\vec{\mu}, \matr{\Sigma}, S; \vec{x}) =
    \begin{cases}
      \frac{1}{\normal(\vec{\mu}, \matr{\Sigma}; S)}
      \cdot \normal(\vec{\mu}, \matr{\Sigma}; \vec{x}) & ~~~ \vec{x} \in S \\
      0                                                                                                                  & ~~~ \vec{x} \not\in S
    \end{cases} .
\end{equation}

  Throughout the paper we assume that all the covariance matrices $\Sigma$ are
full rank. The case where $\Sigma$ is not full rank can be easily detected by
drawing $d$ samples and testing whether they are not linearly independent. In
that case, one can solve the estimation problem in the subspace that those
samples span.
\smallskip

\paragr{Membership Oracle of a Set.} Let $S \subseteq \reals^d$ be a subset of
the $d$-dimensional Euclidean space. A \textit{membership oracle} of $S$ is an
efficient procedure $\memb_S$ that computes the characteristic function of $S$,
i.e. $\memb_S(\vec{x}) = \chara\{\vec{x} \in S\}$.

\subsection{Useful Concentration Results}

The following lemma is useful in cases where one wants to show concentration of
a weighted sum of squares of i.i.d random variables.
\begin{theorem}[Lemma 1 of Laurent and Massart \cite{LaurentM00}.] \label{thm:MassartConcentrationLemma}
    Let $x^{(1)}, \dots, x^{(n)}$ independent identically distributed random
  variables following $\normal(0, 1)$ and let $\vec{a} \in \reals_+^d$. Then,
  the following inequalities hold for any $t \in \reals_+$.
  \[ \Prob \left( \sum_{i = 1}^d a_i \left(\left(x^{(i)}\right)^2 - 1\right) \ge 2 \norm{\vec{a}}_2 \sqrt{t} + 2 \norm{\vec{a}}_{\infty} t \right) \le \exp(-t), \]
  \[ \Prob \left( \sum_{i = 1}^d a_i \left(\left(x^{(i)}\right)^2 - 1\right) \le - 2 \norm{\vec{a}}_2 \sqrt{t}\right) \le \exp(-t). \]
\end{theorem}

The following matrix concentration result is also useful in order to show that the empirical covariance matrix of samples drawn from an identity covariance distribution is itself close to identity in the Frobenius norm.

\begin{theorem}[Corollary 4.12 of Diakonikolas et al. \cite{DKK+16}] \label{thm:DiakonikolasConcentrationLemma}
    Let $\rho, \tau > 0$ and $\vec{x}^{(1)}, \dots, \vec{x}^{(n)}$ be i.i.d
  samples from $\normal(\vec{0}, \matr{I})$. There exist a value
  $\delta_2 = O(\rho\log(1/\rho))$, such that if
  $n = \Omega(\frac{d^2 + \log(1/\tau)}{\delta_2^2})$ then
  \[ \Prob\left( \exists T\subseteq [n]: \vert T\vert\leq \rho n ~~\mathrm{ and }~~ \norm{\sum_{i \in T} \frac{1}{\abs{T}}\vec{x}^{(i)} \vec{x}^{(i) T} - \matr{I}}_F \geq \Omega\left(\delta_2\frac{n}{\vert T\vert}\right) \right) \leq \tau. \]
\end{theorem}

  Using the well known fact that the squared Frobenius norm of a symmetric
matrix is equal to the sum of squares of its eigenvalues, we can obtain a bound
on the $\ell_2$ distance of the vector with entries the eigenvalues of
$\sum_{i \in T} \frac{1}{\abs{T}}\vec{x}^{(i)} \vec{x}^{(i) T}$ to the vector
with $1$ to every entry.

  \section{Stochastic Gradient Descent for Learning Truncated Normals} \label{sec:upper bound}

  In this section, we present and analyze our algorithm for estimating the true
mean and covariance matrix of the normal distribution from which the truncated
samples are drawn. Our algorithm is Projected Stochastic Gradient Descent
(PSGD) on the negative log-likelihood function. The steps of our proof are the
following
\begin{enumerate}
  \item First in Section \ref{sec:strong-convexity} we show that for the set of
    parameters $(\vec{\mu}, \matr{\Sigma})$ for which the set $S$ has at least
    a constant mass the negative log-likelihood function is strongly convex
    with respect to the correct parametrization.
  \item Then in Section \ref{sec:initialization} we show how we can find some
    initial estimates $(\vec{\mu}_0, \matr{\Sigma}_0)$ such that the probability
    mass $\normal(\vec{\mu}_0, \matr{\Sigma}_0)$ is a constant.
  \item In Section \ref{sec:box-parameters} we show how based on
    $(\vec{\mu}_0, \matr{\Sigma}_0)$ we can find a set of parameters $D$ such
    that for every $(\vec{\mu}, \matr{\Sigma}) \in D$ it holds that the
    probability mass $\normal(\vec{\mu}, \matr{\Sigma})$ is at least a constant.
  \item In Section \ref{sec:sgd-final} we put everything together to prove that
    the Projected Stochastic Gradient Descent algorithm with projection set $D$
    converges fast to the true estimates which completes the proof of Theorem
    \ref{thm:main theorem}.
\end{enumerate}

\subsection{Strong-convexity of the Negative Log-Likelihood Function} \label{sec:strong-convexity}

  Let $S \subseteq \reals^d$ be a subset of the $d$-dimensional Euclidean space.
We assume that we have access to $n$ samples $\vec{x}_i$ from
$\normal(\vec{\mu}^*, \vec{\Sigma}^*, S)$. We start by showing that the negative
log-likelihood of the truncated samples is strongly convex as long as we
re-parameterize our problem in terms of $\vec{v}=\matr{\Sigma}^{-1}\vec{\mu}$
and $\matr{T}=\matr{\Sigma}^{-1}$.

\subsubsection{Log-Likelihood for a Single Sample}
\label{sec:logLikelihoodOneSample}

Given one vector $\vec{x} \in \reals^d$, the negative log-likelihood that
$\vec{x}$ is a sample of the form $\normal(\vec{\mu}, \vec{\Sigma}, S)$ is
\begin{align} \label{eq:logLikelihoodOneSample}
  \ell(\vec{\mu}, \matr{\Sigma}; \vec{x}) = & \frac{1}{2} (\vec{x} -
  \vec{\mu})^T \matr{\Sigma}^{-1} (\vec{x} - \vec{\mu}) +
  \log \left( \int_S \exp \left( - \frac{1}{2} (\vec{z} - \vec{\mu})^T
  \matr{\Sigma}^{-1} (\vec{z} - \vec{\mu}) \right) d\vec{z} \right)
  \nonumber \\
  = & \frac{1}{2} \vec{x}^T \matr{\Sigma}^{-1} \vec{x} -
    \vec{x}^T \matr{\Sigma}^{-1} \vec{\mu}
    + \log \left( \int_S \exp \left(
    - \frac{1}{2} \vec{z}^T \matr{\Sigma}^{-1} \vec{z} +
    \vec{z}^T \matr{\Sigma}^{-1} \vec{\mu} \right) d\vec{z} \right)
\end{align}

\noindent Here we will define a different parametrization of the problem with
respect to the variables $\matr{T} = \matr{\Sigma}^{-1}$ and
$\vec{\nu} = \matr{\Sigma}^{-1} \vec{\mu}$, where $\matr{T} \in
\reals^{d \times d}$ and $\vec{\nu} \in \reals^d$. Then the
likelihood function with respect to $\vec{\nu}$ and $\matr{T}$ is equal to
\begin{align} \label{eq:logLikelihoodOneSampleReparam}
  \ell(\vec{\nu}, \matr{T}; \vec{x})
  = & \frac{1}{2} \vec{x}^T \matr{T} \vec{x} - \vec{x}^T \vec{\nu}
      + \log \left( \int_S \exp \left( - \frac{1}{2} \vec{z}^T \matr{T} \vec{z}
      + \vec{z}^T \vec{\nu} \right) d\vec{z} \right)
\end{align}

\noindent We now compute the gradient of $\ell(\vec{\nu}, \matr{T}; \vec{x})$
with respect to the set of variables
$\begin{bmatrix} \matr{T}^{\flat} \\ \vec{\nu} \end{bmatrix}$.

\begin{align} \label{eq:gradientofLikelihoodOneSample}
  \nabla \ell(\vec{\nu}, \matr{T}; \vec{x}) = & - \begin{bmatrix}
                    \left(- \frac{1}{2} \vec{x} \vec{x}^T \right)^{\flat} \\
                    \vec{x}
                  \end{bmatrix}
                  + \frac{\int_S \begin{bmatrix}
                                    \left( - \frac{1}{2} \vec{z}
                                      \vec{z}^T \right)^{\flat} \\
                                    \vec{z}
                                  \end{bmatrix}
                                  \exp \left( - \frac{1}{2} \vec{z}^T \matr{T}
                  \vec{z} + \vec{z}^T \vec{\nu} \right) d\vec{z}}
                  {\int_S \exp \left( - \frac{1}{2} \vec{z}^T \matr{T}
                  \vec{z} + \vec{z}^T \vec{\nu} \right) d\vec{z}}
                  \nonumber \\
              = & - \begin{bmatrix}
                    \left(- \frac{1}{2} \vec{x} \vec{x}^T \right)^{\flat} \\
                    \vec{x}
                  \end{bmatrix}
                  + \frac{\int_S \begin{bmatrix}
                                    \left( - \frac{1}{2} \vec{z}
                                      \vec{z}^T \right)^{\flat} \\
                                    \vec{z}
                                  \end{bmatrix}
                        \exp \left( - \frac{1}{2} \vec{z}^T \matr{T}
                          \vec{z} + \vec{z}^T \vec{\nu} -
                    \norm{\matr{T}^{-1} \vec{\nu}}_2^2 \right) d\vec{z}}
                  {\int_S \exp \left( - \frac{1}{2} \vec{z}^T \matr{T}
                  \vec{z} + \vec{z}^T \matr{B} \vec{x} -
                    \norm{\matr{T}^{-1} \vec{\nu}}_2^2\right) d\vec{z}}
                  \nonumber \\
              = & - \begin{bmatrix}
                    \left(- \frac{1}{2} \vec{x} \vec{x}^T \right)^{\flat} \\
                    \vec{x}
                  \end{bmatrix}
                  + \Exp_{\vec{z} \sim \normal(\matr{T}^{-1} \vec{\nu},
                      \matr{T}^{-1}, S)}
                  \left[
                    \begin{bmatrix}
                      \left( - \frac{1}{2} \vec{z}
                       \vec{z}^T \right)^{\flat} \\
                      \vec{z}
                    \end{bmatrix}
                  \right]
\end{align}

\noindent Finally, we compute the Hessian $\Hessian_{\ell}$ of the
log-likelihood function.

\begin{align} \label{eq:HessianofLikelihoodOneSample}
  \Hessian_{\ell}(\vec{\nu}, \matr{T})
              & = \frac{\int_S \begin{bmatrix}
                                   \left( - \frac{1}{2} \vec{z}
                                     \vec{z}^T \right)^{\flat} \\
                                   \vec{z}
                                 \end{bmatrix}
                                 \begin{bmatrix}
                                   \left( - \frac{1}{2} \vec{z}
                                    \vec{z}^T \right)^{\flat} \\
                                   \vec{z}
                                 \end{bmatrix}^T
                                  \exp \left( - \frac{1}{2} \vec{z}^T \matr{T}
                  \vec{z} + \vec{z}^T \vec{\nu} \right) d\vec{z}}
                  {\int_S \exp \left( - \frac{1}{2} \vec{z}^T \matr{T}
                  \vec{z} + \vec{z}^T \vec{\nu} \right) d\vec{z}}
                  \nonumber \\
                & ~~~~ - \frac{\int_S \begin{bmatrix}
                                     \left( - \frac{1}{2} \vec{z}
                                       \vec{z}^T \right)^{\flat} \\
                                     \vec{z}
                                   \end{bmatrix}
                                  \exp \left( - \frac{1}{2} \vec{z}^T \matr{T}
                  \vec{z} + \vec{z}^T \vec{\nu} \right) d\vec{z}}
                  {\int_S \exp \left( - \frac{1}{2} \vec{z}^T \matr{T}
                  \vec{z} + \vec{z}^T \vec{\nu} \right) d\vec{z}} \cdot
                  \nonumber \\
                & ~~~~~~~~ \cdot \frac{\int_S \begin{bmatrix}
                                     \left( - \frac{1}{2} \vec{z}
                                       \vec{z}^T \right)^{\flat} \\
                                     \vec{z}
                                   \end{bmatrix}^T
                                  \exp \left( - \frac{1}{2} \vec{z}^T \matr{T}
                    \vec{z} + \vec{z}^T \vec{\nu} \right) d\vec{z}}
                    {\int_S \exp \left( - \frac{1}{2} \vec{z}^T \matr{T}
                    \vec{z} + \vec{z}^T \vec{\nu} \right) d\vec{z}}
                    \nonumber \\
              & = \Cov_{\vec{z} \sim \normal(\matr{T}^{-1} \vec{\nu},
                     \matr{T}^{-1}, S)} \left[
                     \begin{bmatrix}
                       \left( - \frac{1}{2} \vec{z}
                        \vec{z}^T \right)^{\flat} \\
                       \vec{z}
                     \end{bmatrix},
                     \begin{bmatrix}
                       \left( - \frac{1}{2} \vec{z}
                         \vec{z}^T \right)^{\flat} \\
                       \vec{z}
                     \end{bmatrix} \right].
\end{align}

Since the covariance matrix of a random variable is always positive
semidefinite, we conclude that $\Hessian_{\ell}(\vec{\nu}, \matr{T})$ is
positive semidefinite everywhere and hence, we have the following lemma.

\begin{lemma} \label{lem:oneSampleLikelihoodIsConcave}\label{lemma:reparameterized convex}
    The function $\ell(\vec{\nu}, \matr{T}^{-1}; \vec{x})$ is convex with
  respect to $\begin{bmatrix} \matr{T}^{\flat} \\ \vec{\nu} \end{bmatrix}$ for
  all $\vec{x} \in \reals^d$.
\end{lemma}

\subsubsection{Negative Log-Likelihood Function in the Population Model} \label{sec:logLikelihoodPopulation}

  The negative log-likelihood function in the population model is equal to

\begin{align} \label{eq:logLikelihoodPopulation}
  \bar{\ell}(\vec{\nu}, \matr{T}) = & ~
      \Exp_{\vec{x} \sim \normal(\vec{\mu^*}, \matr{\Sigma^*}, S)}
      \left[ \frac{1}{2} \vec{x}^T \matr{T} \vec{x} - \vec{x}^T \vec{\nu}
      \right]
      - \log \left( \int_S \exp \left( - \frac{1}{2} \vec{z}^T \matr{T} \vec{z}
          + \vec{z}^T \vec{\nu} \right) d\vec{z} \right)
\end{align}

\noindent Also, using \eqref{eq:gradientofLikelihoodOneSample} we have that

\begin{align} \label{eq:gradientofLikelihoodPopulation}
  \nabla \bar{\ell}(\vec{\nu}, \matr{T}) = & ~
    - \Exp_{\vec{x} \sim \normal(\vec{\mu^*}, \matr{\Sigma^*}, S)}
    \left[
      \begin{bmatrix}
        \left( - \frac{1}{2} \vec{x} \vec{x}^T \right)^{\flat} \\
        \vec{x}
      \end{bmatrix}
    \right] +
    \Exp_{\vec{z} \sim \normal(\matr{T}^{-1} \vec{\nu}, \matr{T}^{-1}, S)}
    \left[
      \begin{bmatrix}
        \left( - \frac{1}{2} \vec{z} \vec{z}^T \right)^{\flat} \\
        \vec{z}
      \end{bmatrix}
    \right].
\end{align}

\noindent Hence from Lemma \ref{lem:oneSampleLikelihoodIsConcave} we have that
$\bar{\ell}$ is a convex function with respect to $\vec{\nu}$ and $\matr{T}$.
Also, from \eqref{eq:gradientofLikelihoodPopulation} we get that the gradient
$\nabla \bar{\ell}(\vec{\nu}, \matr{T})$ is $\vec{0}$, when
$\vec{\mu^*} = \matr{T}^{-1} \vec{\nu}$ and $\matr{\Sigma^*} = \matr{T}^{-1}$.
From this observation together with the convexity of $\bar{\ell}$ we conclude
that the true parameters, $\vec{\nu}^* = \matr{\Sigma}^{* -1} \vec{\mu}^*$
and $\matr{T}^* = \matr{\Sigma}^{* -1}$ maximize the log-likelihood function.

\begin{lemma} \label{lem:optimalityofMLEPopulation}
    Let $\vec{\nu}^* = \matr{\Sigma}^{* -1} \vec{\mu}^*$,
  $\matr{T}^* = \matr{\Sigma}^{* -1}$, then for any $\vec{\nu} \in \reals^d$,
  $\matr{T} \in \reals^{d \times d}$ it holds that
  \[ \bar{\ell}(\vec{\nu}^*, \matr{T}^*) \le \bar{\ell}(\vec{\nu}, \matr{T}). \]
\end{lemma}

\noindent Observe also that the Hessian of $\ell$ is the same as the Hessian
of $\bar{\ell}$. One property of the reparameterized negative log-likelihood
function that is important for our proof is its \textit{strong convexity}.

\begin{definition}[Strong Convexity]
    Let $g : \reals^d \to \reals$, and let $\Hessian_{g}$ be the Hessian of $g$.
  We say that $g$ is $\lambda$-strongly convex if
  $\Hessian_{g}(\vec{x}) \succeq \lambda \matr{I}$ for all
  $\vec{x} \in \reals^d$.
\end{definition}

\noindent Our goal it to prove that $\bar{\ell}(\vec{\nu}, \matr{T})$ is
strongly convex for parameters $\vec{\nu}$ and $\matr{T}$ such that the
probability mass $\normal(\matr{T}^{-1} \vec{\nu}, \matr{T}^{-1}; S)$ is at
least a constant. Then in Section \ref{sec:box-parameters} we prove it is
possible to efficiently find a set of parameters $D$ such that this conditions
holds and also $D$ contains the true parameters. The main idea of the proof is
to use PSGD with projection set $D$ to recover the parameters $\vec{\mu}^*$ and
$\matr{\Sigma}^*$.

We first prove strong concavity for the case $S = \reals^d$. For
this, we need some definitions.

\begin{definition}  \label{def:minimumEigenvalueWithoutTruncation}
    Let $\matr{\Sigma} \in \reals^{d \times d}$ and also
  $\matr{\Sigma} \succ \matr{0}$ with eigenvalues
  $\lambda_1, \dots, \lambda_n$, then
  we define the minimum eigenvalue $\sigma_m(\matr{\Sigma})$ of the fourth
  moment tensor of $\normal(\vec{0}, \matr{\Sigma})$ as
  \[ \sigma_m(\matr{\Sigma}) = \min\left\{ \min_{i, j \in [d]} \lambda_i
        \cdot \lambda_j, \min_{i \in [d]} \lambda_i \right\}. \]
  We also define the quantity $\lambda_m(\vec{\mu}, \matr{\Sigma})$ as follows
  \[ \lambda_m(\vec{\mu}, \matr{\Sigma}) = \min \left\{ \frac{\sigma_m(\matr{\Sigma})}{4}, \frac{\sigma_m(\matr{\Sigma})}{16 \norm{\vec{\mu}}_2^2 + \sqrt{\sigma_m(\matr{\Sigma})}} \right\}. \]
\end{definition}

\begin{lemma}[Strong Convexity without Truncation] \label{lem:strongConcavityWithoutTruncation}
    Let $\Hessian_{\ell}$ be the Hessian of the negative log likelihood
  function $\bar{\ell}(\vec{\nu}, \matr{T})$, when there is no truncation, i.e.
  $S = \reals^d$. If $\matr{T}^{-1} \succ \matr{0}$ then it holds that
  \[ \Hessian_{\ell}(\vec{\nu}, \matr{T}) \succeq \lambda_m(\matr{T}^{-1} \vec{\nu}, \matr{T}^{-1}) \cdot \matr{I}. \]
\end{lemma}

\begin{proof}
    Let $\vec{\mu} = \matr{T}^{-1} \vec{\nu}$ and
  $\matr{\Sigma} = \matr{T}^{-1}$, we have that
  \[ \Hessian_{\ell}(\vec{\nu}, \matr{T}) = \Cov_{\vec{z} \sim \normal(\vec{\mu},
                 \matr{\Sigma})} \left[
                 \begin{bmatrix}
                   \left( - \frac{1}{2} \vec{z}
                    \vec{z}^T \right)^{\flat} \\
                   \vec{z}
                 \end{bmatrix},
                 \begin{bmatrix}
                   \left( - \frac{1}{2} \vec{z}
                     \vec{z}^T \right)^{\flat} \\
                   \vec{z}
                 \end{bmatrix} \right]. \]

   \noindent Next we define the matrix
   \[ \tilde{Q}(\vec{\mu}, \matr{\Sigma}) \triangleq \frac{1}{2} \matr{\Sigma} \otimes \matr{\Sigma} - \frac{\rho}{4} \cdot (\vec{\mu} \vec{\mu}^T) \otimes \matr{\Sigma} - \frac{\rho}{4} \cdot \vec{\mu} \otimes \matr{\Sigma} \otimes \vec{\mu}^T - \frac{\rho}{4} \cdot \vec{\mu}^T \otimes \matr{\Sigma} \otimes \vec{\mu} - \frac{\rho}{4} \cdot \matr{\Sigma} \otimes (\vec{\mu} \vec{\mu}^T) \]
   \noindent where
   $\rho \triangleq \frac{\sqrt{\sigma_m(\matr{\Sigma})}}{4 \norm{\vec{\mu}}_2^2} + 3$
   and we can prove that
   \begin{claim} \label{clm:covarianceWithoutTruncationDominance}
     It holds that
     \begin{align*}
     \Cov_{\vec{z} \sim \normal(\vec{\mu},
                    \matr{\Sigma})} \left[
                    \begin{bmatrix}
                      \left( - \frac{1}{2} \vec{z}
                       \vec{z}^T \right)^{\flat} \\
                      \vec{z}
                    \end{bmatrix},
                    \begin{bmatrix}
                      \left( - \frac{1}{2} \vec{z}
                        \vec{z}^T \right)^{\flat} \\
                      \vec{z}
                    \end{bmatrix} \right] \succeq
                    \begin{bmatrix}
                      \tilde{Q}(\vec{\mu}, \matr{\Sigma}) & \matr{0} \\
                      \matr{0} & \frac{\rho - 3}{1 + \rho} \cdot \matr{\Sigma}
                    \end{bmatrix}.
     \end{align*}
   \end{claim}

   Hence the eigenvalues of the Hessian $\Hessian_{\ell}$ when $S = \reals^d$
   can be lower bounded by the eigenvalues of
   $\tilde{Q}(\vec{\mu}, \matr{\Sigma})$ and the eigenvalues of
   $\frac{\rho - 3}{1 + \rho} \matr{\Sigma}$. For this reason we need the
   following claim
   \begin{claim} \label{clm:covarianceWithoutTruncationEigenvalues}
     It holds that
     $\tilde{Q}(\vec{\mu}, \matr{\Sigma}) \succeq \frac{\sigma_m(\matr{\Sigma})}{4} \cdot \matr{I}_{d^2}$ and $\frac{\rho - 3}{1 + \rho} \cdot \matr{\Sigma} \succeq \frac{\sigma_m(\matr{\Sigma})}{16 \norm{\vec{\mu}}_2^2 + \sqrt{\sigma_m(\matr{\Sigma})}} \cdot \matr{I}_d$.
   \end{claim}

   \noindent If we combine Claim \ref{clm:covarianceWithoutTruncationDominance}
   and Claim \ref{clm:covarianceWithoutTruncationEigenvalues} then the lemma
   follows. We present the proof of these claims in the Appendix
   \ref{app:upper bound}.
\end{proof}

  To prove strong convexity in the presence of truncation, we use the
following anticoncentration bound of the Gaussian measure on sets characterized
by polynomial threshold functions.

\begin{theorem}[Theorem 8 of \cite{CarberyW01}] \label{thm:GaussianMeasurePolynomialThresholdFunctions}
    Let $q, \gamma \in \reals_+$, $\vec{\mu} \in \reals^d$,
  $\matr{\Sigma} \in \symm_d$ and $p : \reals^d \to \reals$ be a multivariate
  polynomial of degree at most $k$, we define
  \[ \bar{S} = \left\{ \vec{x} \in \reals^d \mid \abs{p(\vec{x})} \le \gamma
     \right\}, \]
  then there exists an absolute constant $C$ such that
  \[ \normal(\vec{\mu}, \matr{\Sigma}; \bar{S}) \le \frac{C q \gamma^{1/k}}
     {\left( \Exp_{\vec{z} \sim \normal(\vec{\mu}, \matr{\Sigma})}
      \left[ \abs{p(\vec{z})}^{q/k} \right] \right)^{1/q}}. \]
\end{theorem}

  Using the above anticoncentration theorem we can utilize Lemma
\ref{lem:strongConcavityWithoutTruncation} and prove that in the presence of
truncation, for which at least a constant fraction of the data survives, the
eigenvalues of the Hessian cannot be much lower than the eigenvalues in the
non-truncated case.

\begin{lemma}[Strong Convexity with Truncation]
\label{lem:strongConcavityofLogLikelihoodWithTruncation}\label{lemma:strong convexity}
    Let $\Hessian_{\ell}$ be the Hessian of the negative log likelihood
  function $\bar{\ell}(\vec{\nu}, \matr{T})$, with the presence of arbitrary
  truncation $S$ such that $\normal(\vec{\mu}, \matr{\Sigma} ; S) \ge \beta$
  for some $\beta \in (0, 1]$, where $\vec{\mu} = \matr{T}^{-1} \vec{\nu}$, and
  $\matr{\Sigma} = \matr{T}^{-1}$. Then it holds that
  \[ \Hessian_{\ell}(\vec{\nu}, \matr{T}) \succeq \frac{1}{2^{13}} \left( \frac{\beta}{C}
  \right)^4 \lambda_m(\vec{\mu}, \matr{\Sigma})
  \cdot \matr{I}, \]
  \noindent where $C$ is the universal constant guaranteed to exist by Theorem
  \ref{thm:GaussianMeasurePolynomialThresholdFunctions}.
\end{lemma}

\begin{proof}
    We first define the vector valued function
  \[\vec{v}(\vec{z}) = \begin{bmatrix}
    \left( - \frac{1}{2} \vec{z}
     \vec{z}^T \right)^{\flat} \\
    \vec{z}
  \end{bmatrix}.\]
  Next we define the matrices $\matr{R}$, $\matr{R}'$ and
   $\matr{R}^*$ as follows
  \begin{align*}
    \matr{R} = & ~ \Cov_{\vec{z} \sim \normal(\vec{\mu}, \matr{\Sigma}, S)} \left[
      \vec{v}(\vec{z}), \vec{v}(\vec{z})
    \right] \\
  = & ~ \Exp_{\vec{z} \sim \normal(\vec{\mu}, \matr{\Sigma}, S)} \left[
       \left(
        \vec{v}(\vec{z})
       - \Exp_{\vec{z} \sim \normal(\vec{\mu}, \matr{\Sigma}, S)} \left[
           \vec{v}(\vec{z})
         \right]
       \right) \cdot \left(
       \vec{v}(\vec{z}) -
       \Exp_{\vec{z} \sim \normal(\vec{\mu}, \matr{\Sigma}, S)} \left[
           \vec{v}(\vec{z})
         \right]
       \right)^T \right], \\
  \matr{R}' = & ~ \Exp_{\vec{z} \sim \normal(\vec{\mu}, \matr{\Sigma})} \left[
     \left(
      \vec{v}(\vec{z})
     - \Exp_{\vec{z} \sim \normal(\vec{\mu}, \matr{\Sigma}, S)} \left[
         \vec{v}(\vec{z})
       \right]
     \right) \cdot \left(
       \vec{v}(\vec{z}) -
       \Exp_{\vec{z} \sim \normal(\vec{\mu}, \matr{\Sigma}, S)} \left[
         \vec{v}(\vec{z})
       \right]
     \right)^T \right], \\
    \matr{R}^* = & ~ \Exp_{\vec{z} \sim \normal(\vec{\mu}, \matr{\Sigma})}  \left[
     \left(
      \vec{v}(\vec{z})
     - \Exp_{\vec{z} \sim \normal(\vec{\mu}, \matr{\Sigma})} \left[
         \vec{v}(\vec{z})
       \right]
     \right) \cdot \left(
     \vec{v}(\vec{z}) -
     \Exp_{\vec{z} \sim \normal(\vec{\mu}, \matr{\Sigma})} \left[
         \vec{v}(\vec{z})
       \right]
     \right)^T \right].
  \end{align*}
  \noindent Our goal is to lower bound the eigenvalues of $\matr{R}$ based on
  the eigenvalues of $\matr{R}^*$ which we can lower bound using Lemma
  \ref{lem:strongConcavityWithoutTruncation}. To do this we use $\matr{R}'$ as
  an intermediate step and then we use Theorem
  \ref{thm:GaussianMeasurePolynomialThresholdFunctions} to relate the
  eigenvalues of $\matr{R}$ with the eigenvalues of $\matr{R}'$. So the first
  step is to relate the eigenvalues of $\matr{R}'$ with the eigenvalues of
  $\matr{R}^*$ in the following claim.

  \begin{claim} \label{clm:eigenvaluesofRprimevsRstar}
    It holds that $\matr{R}' \succeq \matr{R}^*$.
  \end{claim}

  \noindent Now let $\vec{v} \in \reals^d$ and
  $\matr{U} \in \reals^{d \times d}$ with
  $\norm{\vec{v}}_2^2 + \norm{\matr{U}}_F^2 = 1$. We have that
  \begin{align*}
    \begin{bmatrix} \left(\matr{U}^{\flat}\right)^T & \vec{v}^T \end{bmatrix}
      \matr{R}
    \begin{bmatrix} \matr{U}^{\flat} \\ \vec{v} \end{bmatrix}
      =
    \Exp_{\vec{z} \sim \normal(\vec{\mu}, \matr{\Sigma}, S)} \left[
      p_{(\matr{U}, \vec{v})}(\vec{z})
    \right]
  \end{align*}
  \begin{align*}
    \begin{bmatrix} \left(\matr{U}^{\flat}\right)^T & \vec{v}^T \end{bmatrix}
      \matr{R}'
    \begin{bmatrix} \matr{U}^{\flat} \\ \vec{v} \end{bmatrix}
      =
    \Exp_{\vec{z} \sim \normal(\vec{\mu}, \matr{\Sigma})} \left[
      p_{(\matr{U}, \vec{v})}(\vec{z})
    \right]
  \end{align*}
  \begin{align*}
    \begin{bmatrix} \left(\matr{U}^{\flat}\right)^T & \vec{v}^T \end{bmatrix}
      \matr{R}^*
    \begin{bmatrix} \matr{U}^{\flat} \\ \vec{v} \end{bmatrix}
      =
    \Exp_{\vec{z} \sim \normal(\vec{\mu}, \matr{\Sigma})} \left[
      p^*_{(\matr{U}, \vec{v})}(\vec{z})
    \right]
  \end{align*}
  \noindent where $p_{(\matr{U}, \vec{v})}(\vec{z})$,
  $p_{(\matr{U}, \vec{v})}^*(\vec{z})$ are polynomials of degree
  at most $4$ whose coefficients depend on $\matr{U}$ and $\vec{v}$. Also,
  observe that for every $\vec{z} \in \reals^d$ we have that
  $p_{(\matr{U}, \vec{v})}(\vec{z}) \ge 0$ and
  $p_{(\matr{U}, \vec{v})}^*(\vec{z}) \ge 0$. From Claim
  \ref{clm:eigenvaluesofRprimevsRstar} we get that
  $\Exp_{\vec{z} \sim \normal(\vec{\mu}, \matr{\Sigma})} \left[ p_{(\matr{U}, \vec{v})}(\vec{z}) \right] \ge \Exp_{\vec{z} \sim \normal(\vec{\mu}, \matr{\Sigma})} \left[ p^*_{(\matr{U}, \vec{v})}(\vec{z}) \right]$
  and from Lemma \ref{lem:strongConcavityWithoutTruncation} we get that
  $\Exp_{\vec{z} \sim \normal(\vec{\mu}, \matr{\Sigma})} \left[ p^*_{(\matr{U}, \vec{v})}(\vec{z}) \right] \ge \lambda_m(\vec{\mu}, \matr{\Sigma})$
  and therefore it holds that
  \begin{equation} \label{eq:proofOfStrongConcavityWithTruncation1}
    \Exp_{\vec{z} \sim \normal(\vec{\mu}, \matr{\Sigma})} \left[ p_{(\matr{U}, \vec{v})}(\vec{z}) \right] \ge \lambda_m(\vec{\mu}, \matr{\Sigma}).
  \end{equation}

  \noindent What is left is to lower bound
  $\Exp_{\vec{z} \sim \normal(\vec{\mu}, \matr{\Sigma}, S)} \left[ p_{(\matr{U}, \vec{v})}(\vec{z}) \right]$
  with respect to
  $\Exp_{\vec{z} \sim \normal(\vec{\mu}, \matr{\Sigma})} \left[ p_{(\matr{U}, \vec{v})}(\vec{z}) \right]$.
  For this purpose we are going to use Theorem
  \ref{thm:GaussianMeasurePolynomialThresholdFunctions} and the fact that
  $\normal(\vec{\mu}, \matr{\Sigma}; S) \ge \beta$. We define
  \[ \gamma = \left( \frac{1}{C} \frac{\beta}{8} \right)^4 \lambda_m(\vec{\mu}, \matr{\Sigma}) ~~ \text{ and } ~~ \bar{S} = \{ \vec{x} \in \reals^d \mid p_{(\matr{U}, \vec{v})}(\vec{x}) \le \gamma \} \]
  \noindent and applying Theorem
  \ref{thm:GaussianMeasurePolynomialThresholdFunctions} together with
  \eqref{eq:proofOfStrongConcavityWithTruncation1}
  we get that $\normal(\vec{\mu}, \matr{\Sigma}; \bar{S}) \le \frac{\beta}{2}$.
  Therefore, when calculating the expectation
  $\Exp_{\vec{z} \sim \normal(\vec{\mu}, \matr{\Sigma}, S)}
  \left[
    p_{(\matr{U}, \vec{v})}(\vec{z})
  \right]$ at least the half of the mass of the mass of
  $\normal(\vec{\mu}, \matr{\Sigma}, S)$ is in points $\vec{z}$ such that
  $\vec{z} \not\in \bar{S}$. This implies that
  \[ \Exp_{\vec{z} \sim \normal(\vec{\mu}, \matr{\Sigma}, S)} \left[ p_{(\matr{U}, \vec{v})}(\vec{z}) \right] \ge \frac{1}{2} \gamma = \frac{1}{2^{13}} \left( \frac{\beta}{C} \right)^4 \lambda_m(\vec{\mu}, \matr{\Sigma}), \]
  \noindent and the lemma follows.
\end{proof}

\subsection{Initialization with Empirical Mean and Empirical Covariance Matrix} \label{sec:initialization}

  In order to efficiently optimize the negative log-likelihood and maintain its
strong-convexity we need to search over a set of parameters that assign
significant measure to the truncation set we consider. In addition, we need
that the initial point of our algorithm lies in that set and satisfies this
condition.

  Before defining the appropriate set of parameters for the Projected Gradient
Descent algorithm, we prove that a good initialization for PSGD is the
empirical mean and the empirical covariance matrix of the truncated distribution
$\normal(\vec \mu,\matr \Sigma,S)$.

  We begin by showing that only few truncated samples suffice to obtain
accurate estimates $ \hat{\vec{\mu}}_S$ and $\hat{\matr{\Sigma}}_S$ of the mean
and covariance.

\begin{lemma}[Concentration of Empirical Mean and Empirical Covariance] \label{lem:conditional estimation}
    Let $(\vec{\mu}_S, \matr{\Sigma}_S)$ be the mean and covariance of the
  truncated Gaussian $\normal(\vec \mu,\matr \Sigma,S)$ where
  $S \subseteq \reals^d$ such that $\normal(\vec \mu,\matr \Sigma;S) = \alpha$.
  Using $\tilde O(\frac{d}{\eps^2} \log(1/\alpha) \log^2(1/\delta))$ samples,
  we can compute estimates $ \hat{\vec{\mu}}_S$ and $\hat{\matr{\Sigma}}_S$ such
  that
  \[ \norm{\matr{\Sigma}^{-1/2} (\hat{\vec{\mu}}_S - \vec{\mu}_S)}_2 \le \eps \quad \mathrm{ and } \quad (1-\eps) \matr{\Sigma}_S \preceq \hat{\matr{\Sigma}}_S \preceq (1+\eps) \matr{\Sigma}_S \]
  with probability at least $1 - \delta$.
\end{lemma}

\begin{proof}
    Let $\{\vec{x}^{(i)}\}_{i = 1}^n$ be the set of $n$ i.i.d samples drawn
  from the truncated Gaussian
  $\normal(\vec \mu,\matr \Sigma,S)$. Let also
  $\hat{\vec{\mu}}_S = \frac{1}{n} \sum_{i = 1}^n \vec{x}^{(i)}$ be the
  empirical mean and
  $\hat{\matr{\Sigma}}_S = \frac{1}{n} \sum_{i = 1}^n (\vec x^{(i)} - \hat{\vec{\mu}}_S)(\vec{x}^{(i)} - \hat{\vec{\mu}}_S)^T$
  be the empirical covariance matrix.

    If we show the desired inequalities for
  $(\vec \mu,\matr \Sigma) = (\vec 0, \matr I)$, then we the general case
  follows too by applying a simple affine transformation to the samples
  $\vec{x}^{(i)}$ and the set $S$. Hence we assume without loss of generality
  that $(\vec \mu,\matr \Sigma) = (\vec 0, \matr I)$.

    The $n$ samples $\vec{x}^{(i)}$ drawn from $\normal(\vec{0}, \matr{I}, S)$,
  can be seen as $O(n/\alpha)$ samples from $\normal(\vec{0}, \matr{I})$ where
  we only keep those that follows inside the set $S$. Using the well known
  concentration of measure result about the maximum of $K$ standard normal
  variables we get that with probability at least $1 - \delta$, for all samples
  $i$ and coordinates $j$, it holds that
  $\abs{x^{(i)}_j} \le \log\left(\frac{n d}{\alpha \delta}\right)$. Hence if
  we condition on that event we can use Hoeffding's inequality and we get that
  \[ \Prob\left( \abs{\hat{\vec \mu}_{S,j}-\vec \mu_{S,j}} \geq \frac{\eps}{\sqrt{d}} \right) \leq 2 \exp\left(-\frac{n \eps^2}{ d \log(1/\alpha\delta)} \right). \]
  Therefore, if
  $n \ge \Omega(\frac{d \log(n d/\alpha\delta) \log(1/\delta)}{\varepsilon^2})$,
  $n$ samples are sufficient to learn $\vec{\mu}_S$ with error $\eps$ with
  probability $1 - \delta$.
  \smallskip

    The same way we can use the matrix concentration inequality Corollary 5.52
  of \cite{vershynin2010introduction} to get that with probability at least
  $1 - \delta$ is holds that
  \[ (1 - \eps) \matr{\Sigma}_S \preceq \hat{\matr{\Sigma}}_S \preceq (1 + \eps) \matr{\Sigma}_S \]
  provided that $n \ge \Omega(d/\eps^2 \log(n d/\alpha\delta) \log(d/\delta))$.
  The lemma then follows if we set $n$ to be
  $\tilde \Theta(d/\eps^2 \log^2(1/\alpha\delta))$.
\end{proof}

  Our next goal is to show that the distance of the estimates
$\hat{\vec{\mu}}_S$ and $\hat{\matr{\Sigma}}_S$ to the true parameters
$\vec{\mu}, \matr{\Sigma}$ is a constant that depends only on the mass $\alpha$
of the set $S$. We do this by first considering the distance between
$\vec{\mu}, \matr{\Sigma}$ and the true mean and covariance of the truncated
Gaussian distribution $\normal(\vec{\mu}, \matr{\Sigma}, S)$ which we denote by
${\vec{\mu}}_S$ and ${\matr{\Sigma}}_S$.

\begin{lemma}[Truncated vs Non-truncated Parameters] \label{lem:conditional mean closeness} \label{lemma:initialization satisfies sufficient}
    Let $(\vec{\mu}_S, \matr{\Sigma}_S)$ be the mean and covariance matrix of
  the truncated Gaussian $\normal(\vec{\mu}, \matr{\Sigma}, S)$ with
  $\normal(\vec{\mu}, \matr{\Sigma}; S) = \alpha$. Then the following
  statements hold
  \begin{enumerate}
  \item $ \norm{\vec{\mu}_S - \vec{\mu}}_{\matr{\Sigma}}
           \leq O\left(\sqrt{\log \frac{1}{\alpha}}\right)$,
  \item ${\matr{\Sigma}}_S \succeq \Omega(\alpha^2) \matr{\Sigma}$, and
  \item $\norm{\matr \Sigma^{-1/2} \matr{\Sigma}_S \matr{\Sigma}^{-1/2} - \matr{I}}_F \le O\left(\log \frac{1}{\alpha}\right)$.
\end{enumerate}
\end{lemma}

\begin{proof}
    As in the proof of Lemma \ref{lem:conditional estimation} we can assume that
  $(\vec{\mu}, \matr{\Sigma}) = (\vec{0}, \matr{I})$ and then the general case
  follows by applying an affine transformation to $\vec{\mu}$, $\matr{\Sigma}$
  and $S$. Hence we assume without loss of generality that
  $\vec{\mu} = \vec{0}$, $\matr{\Sigma} = \matr{I}$. Once we have established
  that we have an additional freedom to transform the space. In particular,
  after applying the aforementioned affine transformation we have that
  $\vec{\mu} = \vec{0}$ and $\matr{\Sigma} = \matr{I}$, hence if we apply any
  additional \textit{unitary transformation} of the space we still have
  $\vec{\mu} = \vec{0}$ and $\matr{\Sigma} = \matr{I}$. Additionally we know
  that since $\Sigma_S$ is a symmetric matrix, it can be diagonalized by a
  unitary matrix. Therefore we can apply also this transformation of the space
  that does not change anything in the results of this lemma since it is
  unitary. We conclude that we can assume without loss of generality that
  $\vec{\mu} = \vec{0}$, $\matr{\Sigma} = \matr{I}$ and that $\matr{\Sigma}_S$
  is a diagonal matrix with entries $\lambda_1 \le \dots \le \lambda_d$.
  \medskip

  \noindent \textbf{Proof of 1.} We will show that
  \[ \norm{\vec{\mu}_S}_2 \leq \sqrt{2\log \frac{1}{\alpha}} + 1 \]
  which implies 1. for arbitrary $\vec{\mu}, \matr{\Sigma}$ after applying the
  standard transformation that we discussed in the beginning of the proof.
  Consider the direction of the sample mean $\vec{\hat{\mu}}_S$. The worst
  case subset $S \subset \reals^d$ of mass at least $\alpha$ that would
  maximize
  $\norm{\vec{\mu}_S}_2 = \frac{\Exp_{\vec{x} \sim \normal(\vec{0}, \matr{I})}\left[ \vec{1}_{\vec{x} \in S} \vec{x}^T \vec{\hat{\mu}}_S\right]}{\Exp_{\vec{x} \sim N(\vec{0},  \matr{I})}\left[\vec{1}_{\vec{x} \in S}\right]}$
  is the following:
  \[ S = \left\{ \vec{x}^T \vec{\hat{\mu}}_S > F^{-1}(1 - \alpha) \right\} \]
  where $F$ is the CDF of the standard normal distribution. Since
  $\alpha = 1 - F(t) \leq e^{-\frac{t^2}{2}}$, we have that
  $t \leq \sqrt{2\log(\frac{1}{\alpha})}$. The bound follows for the simple
  inequality $\Exp_{x \sim N(0, 1)}[x \vert x \geq t] \leq 1 + t$.
  \medskip

  \noindent \textbf{Proof of 2.} We want to bound the expectation
  $\lambda_1 = \Exp_{\vec{x}\sim \normal(\vec{0},\matr{I},S)}[(x_1 - \mu_{S,1})^2]$.
  Since $\normal(\vec{0},\matr{I};S) = \alpha$, the worst case set, i.e the one
  that minimizes $\lambda_1$, is the one that has $\alpha$ mass as close as
  possible to the hyperplane $x_1 = \mu_{S,1}$. However, the maximum mass that
  the $\normal(\vec{0}, \matr{I})$ Gaussian places at the set
  $\{x_1 \mid \abs{x_1 - \mu_{S,1}} < c\}$ is at most $2c$ as the density of
  the standard univariate Gaussian $\normal(0, 1)$ is at most $1$. Thus the
  $\Exp_{\vec{x} \sim \normal(\vec{0},\matr{I},S)}[(x_1 - \mu_{S,1})^2]$ is at
  least the variance of the uniform distribution $U[-\alpha/2, \alpha/2]$ which
  is $\alpha^2/12$. Thus $\lambda_i \ge \lambda_1 \ge \alpha^2/12$.
  \medskip

  \noindent \textbf{Proof of 3.} Finally, case 3, follows from
  Theorem~\ref{thm:DiakonikolasConcentrationLemma}. Consider any large set
  $\{\vec{x}^{(i)}\}_{i = 1}^n$ of $n$ samples from
  $\normal(\vec{\mu}, \matr{\Sigma})$. Theorem
  \ref{thm:DiakonikolasConcentrationLemma}, implies that with probability
  $1 - o(1/n)$, for all $T \subseteq [n]$ with $\abs{T} = \Theta( \alpha n )$,
  we have that
  $\norm{\sum_{i \in T} \frac{1}{\vert T\vert }\vec{x}^{(i)} \vec{x}^{(i) T} -\matr{I}}_F \geq \Omega\left( \log(1/\alpha) \right)$.
  In particular, the same is true for the set of $\Theta( \alpha n )$ samples
  that lie in the set $S$. As $n \rightarrow \infty$ the empirical second
  moment $\sum_{ib\in T} \frac{1}{\vert T\vert} \vec{x}^{(i)} \vec{x}^{(i) T}$
  converges to $\matr{\Sigma}_S + \vec{\mu}_S \vec{\mu}_S^T$. We thus obtain
  that
  $\norm{\matr{\Sigma}_S + \vec \mu_S \vec \mu_S^T - \matr{I}}_F \le \sqrt{ O\left( \log(1/\alpha) \right) }$,
  which implies that
  $\norm{\matr{\Sigma}_S - \matr{I}}_F \le \sqrt{ O\left( \log(1/\alpha) \right) } + \vec \mu_S \vec \mu_S^T \le O\left( \log(1/\alpha) \right)$.
\end{proof}

\noindent An immediate corollary of Lemma \ref{lem:conditional estimation} and
Lemma \ref{lem:conditional mean closeness} combined with
Theorem \ref{thm:DiakonikolasConcentrationLemma} is the following.

\begin{corollary}[Empirical Parameters vs True Parameters] \label{corollary:initialization}
    The empirical mean $\hat{\vec{\mu}}_S$ and covariance
  $\hat{\matr{\Sigma}}_S$ computed using $\tilde O(d^2 \log^2(1/\alpha\delta))$
  samples from a truncated Normal $\normal(\vec \mu,\matr \Sigma, S)$ with
  $\normal(\vec \mu,\matr \Sigma;S) = \alpha$ satisfies with probability at
  least $1 -\delta$ the following
  \begin{enumerate}
    \item $\norm{\hat{\vec{\mu}}_S - \vec{\mu}}_{\matr \Sigma} \leq O(\sqrt{\log \frac{1}{\alpha}})$,
    \item ${\hat{\matr{\Sigma}}}_S \succeq \Omega(\alpha^2) \matr{\Sigma}$,
    \item $\norm{\matr \Sigma^{-1/2} \matr{\hat{\Sigma}}_S \matr \Sigma^{-1/2}  - \matr I}_F \le O(\log \frac{1}{\alpha} )$.
  \end{enumerate}
\end{corollary}

\begin{proof}
    The first two properties follow by applying Lemma
  \ref{lem:conditional estimation} with $\eps = 1/2$ to Lemma
  \ref{lem:conditional mean closeness}. The last one follows by Theorem
  \ref{thm:DiakonikolasConcentrationLemma}, using an identical argument to the
  proof of part 3. of Lemma \ref{lem:conditional mean closeness}. Note that the
  required sample complexity has a quadratic dependence on $d$ as it is
  necessary for closeness in Frobenius norm, and as it is required by Theorem
  \ref{thm:DiakonikolasConcentrationLemma}.
\end{proof}

\subsection{A Set of Parameters with Non-Trivial Mass in the Truncation Set} \label{sec:box-parameters}

  In the previous section, in Corollary \ref{corollary:initialization}, we
showed that the empirical mean $\hat{\vec{\mu}}_S$ and the empirical covariance
matrix $\hat{\matr{\Sigma}}_S$ of the truncated Gaussian distribution
$\normal(\vec{\mu}^*, \matr{\Sigma}^*, S)$ using
$n = \tilde{O}\left( d^2 \right)$ samples are close to the true parameters
$\vec{\mu}^*, \matr{\Sigma}^*$.

  In this section, in Lemma \ref{lemma:large mass}, we show that these
guarantees are sufficient to show that the Gaussian distribution
$\normal(\hat{\vec{\mu}}_S, \hat{\matr{\Sigma}}_S)$ assigns constant mass to
the set $S$. This way we can define, based on $\hat{\vec{\mu}}_S$ and on
$\hat{\matr{\Sigma}}_S$, a convex set of parameters with the property that
every Gaussian distribution with parameters in this set assigns constant mass
to the set $S$. This set of parameters is the one that we use to run the
Projected Stochastic Gradient Descent algorithm in Section \ref{sec:sgd-final}
to prove our Theorem \ref{thm:main theorem}. We begin with the statement and the
proof of Lemma \ref{lemma:large mass}.

\begin{lemma}\label{lemma:large mass}\label{lemma:sufficient constraint}
    Consider two Gaussian distributions $\normal(\vec{\mu}_1, \matr{\Sigma}_1)$
  and $\normal(\vec{\mu}_2, \matr{\Sigma}_2)$, such that for some
  $B \in \reals_+$
  \begin{enumerate}
    \item $\norm{\matr I - \matr \Sigma_1^{1/2} \matr \Sigma_2^{-1} \matr \Sigma_1^{1/2}}_F \le B$,
    \item $\frac{1}{B} \cdot \matr{I} \preceq \matr{\Sigma}_1^{-1/2} \matr{\Sigma}_2 \matr{\Sigma}_1^{-1/2} \preceq B \cdot \matr{I}$,
    \item $\norm{\matr \Sigma_2^{-1} \matr \Sigma_1^{1/2} (\vec \mu_1 - \vec \mu_2)}_2 \le B$.
  \end{enumerate}
  Suppose that for a set $S \subseteq \reals^d$ we have that
  $\normal(\vec{\mu}_1, \matr{\Sigma}_1; S) \ge \alpha$. Then
  $\normal(\vec{\mu}_2, \matr{\Sigma}_2; S) \ge (\alpha/12)^{23 B^5}$.
\end{lemma}

\begin{proof}
    Using the same argument as in the beginning of the proof of Lemma
  \ref{lem:conditional mean closeness} we may assume without loss of
  generality that $(\vec \mu_1, \matr \Sigma_1) = (\vec 0, \matr I)$ and that
  $\matr \Sigma_2 = \diag(\lambda_1, \dots, \lambda_d)$. Also, for simplicity
  we use $\vec{\mu}$ to denote $\vec{\mu}_2$. Hence the assumptions of the lemma
  can be rewritten as
  \begin{enumerate}
    \item $\sum_i (1 - 1/\lambda_i)^2 < B^2$,
    \item $1/B < \lambda_i < B$, and
    \item $\sqrt{\sum_i (\mu_i^2/\lambda_i)} < B$.
  \end{enumerate}
  \noindent The first two bounds also imply that
  \begin{align} \label{eq:proof:lemma:large mass:firstAssumption}
    \sum_i (1 - \lambda_i)^2 < B^4.
  \end{align}
  \noindent We begin the proof of the lemma by observing that
  \begin{align} \label{eq:proof:lemma:large mass:massRatioExpression}
    \normal(\vec{\mu}_2, \matr{\Sigma}_2; S) = \Exp_{\vec{x} \sim \normal(\vec{\mu}_1, \matr{\Sigma}_1)}\left[ \chara\{\vec{x} \in S\} \cdot \frac{\normal(\vec{\mu}_2, \matr{\Sigma}_2; x)}{\normal(\vec{\mu}_1, \matr{\Sigma}_1; x)} \right].
  \end{align}
  \noindent Also via simple calculations we have that
  \begin{align} \label{eq:proof:lemma:large mass:massRatioExpansion}
    \frac{\normal(\vec{\mu}_2, \matr{\Sigma}_2; \vec{x})}{\normal(\vec{\mu}_1, \matr{\Sigma}_1; \vec{x})} & = \exp\left( -\sum_i \left[\frac{(x_i - \mu_i)^2}{\lambda_i} - x_i^2 + \log \lambda_i\right] \right).
  \end{align}

    Our next step is to show that for a random
   $\vec{x} \sim \normal(\vec{\mu}_1, \matr{\Sigma}_1)$ with probability higher
  than $1 - \alpha/2$ the ratio
  $\frac{\normal(\vec{\mu}_2, \matr{\Sigma}_2; \vec{x})}{\normal(\vec{\mu}_1, \matr{\Sigma}_1; \vec{x})}$
  is larger than some bound $T$. This implies that
  \begin{align} \label{eq:proof:lemma:large mass:massRatioLowerBound}
  \normal(\vec \mu_2, \matr \Sigma_2; S) = \Exp_{\vec x \sim \normal(\vec \mu_1, \matr \Sigma_1)} \left[ \chara\{\vec{x} \in S\} \cdot \frac{\normal(\vec \mu_2, \matr \Sigma_2; x)}{\normal(\vec \mu_1, \matr \Sigma_1; x)} \right] \ge \frac{a}{2} \cdot T.
  \end{align}
  \noindent To obtain the bound $T$, we make the following observations based
  on our assumptions 1. - 3.
  \begin{enumerate}[label=(\alph*)]
    \item if $\lambda_i < 1/2$, then
      $\frac{(x_i - \mu_i)^2}{\lambda_i} - x_i^2 + \log \lambda_i \le B x_i^2 - 2 \frac{x_i \mu_i}{\lambda_i} + \frac{\mu_i^2}{\lambda_i}$,
    \item if $\lambda_i > 2$, then
      $\frac{(x_i - \mu_i)^2}{\lambda_i} - x_i^2 + \log \lambda_i < \frac{\mu_i^2}{\lambda_i} - 2 \frac{x_i \mu_i}{\lambda_i} + \log B$, and
    \item if $\lambda_i \in [1/2, 2]$, then
      $\frac{(x_i - \mu_i)^2}{\lambda_i} - x_i^2 + \log \lambda_i < \left(\frac{1}{\lambda_i} - 1\right) x_i^2 + \frac{\mu_i^2}{\lambda_i} - 2 \frac{x_i \mu_i}{\lambda_i} + \log \lambda_i$.
  \end{enumerate}
  \noindent By \eqref{eq:proof:lemma:large mass:firstAssumption} we have
  $\|\vec \lambda - \vec 1\|_2^2 \le B^4$ and hence the total number of
  eigenvalues $\lambda_i$ that satisfy $\lambda_i \not\in [1/2, 2]$ is at most
  $4 B^4$.
  \medskip

    If we group together all $i$ with $\lambda_i < 1/2$ then using Theorem
  \ref{thm:MassartConcentrationLemma} we have that
  \begin{align} \label{eq:proof:lemma:large mass:smallEigenvalues}
    \Prob_{\vec{x} \sim \normal(\vec{0}, \matr{I})} \left( \sum_{i : \lambda_i < 1/2} B x_i^2 > 8 B^5 \log(6/\alpha) \right) \le \frac{\alpha}{6}
  \end{align}
  \noindent where we have also used the fact that the summation has at most
  $4 B^4$ terms.
  \smallskip

  \noindent We also have the following inequality
  \begin{align} \label{eq:proof:lemma:large mass:outsideEigenvalues}
    \Prob_{\vec{x} \sim \normal(\vec{0}, \matr{I})} \left( - 2 \sum_i \frac{x_i \mu_i}{\lambda_i} \ge B^{3/2} \log(6/\alpha) \right) \le \frac{\alpha}{6}
  \end{align}
  \noindent where we have used concentration of measure for the random variable
  $\vec{w}^T \vec{x}$, with $w_i = \frac{\mu_i}{\lambda_i}$, that follows a
  single dimensional Gaussian distribution with mean $0$ and variance
  $\norm{\vec{w}}_2^2$. To bound $\norm{\vec{w}}_2$ in
  \eqref{eq:proof:lemma:large mass:outsideEigenvalues} we have used the third
  assumption of the lemma that implies
  \[ \norm{\vec{w}}_2 = \sqrt{\sum_i \frac{\mu_i^2}{\lambda_i^2}} \le \sqrt{B} \left( \sqrt{\sum_i \frac{\mu_i^2}{\lambda_i}} \right) \le B^{3/2}. \]

  \noindent If we group together all $i$ with $\lambda_i \in [1, 2]$ then using
  Theorem \ref{thm:MassartConcentrationLemma} with
  $a_i = 1 - \frac{1}{\lambda_i}$ we have that
  \begin{align} \label{eq:proof:lemma:large mass:mediumSmallEigenvalues}
    \Prob_{\vec{x} \sim \normal(\vec{0}, \matr{I})} \left( \sum_{i : \lambda_i \in [1,2]} \left(\frac{1}{\lambda_i} - 1\right) (x_i^2 - 1) \ge 2 \sqrt{\sum_{i : \lambda_i \in [1,2]} \left(\frac{1}{\lambda_i} - 1\right)^2 \log\frac{12}{\alpha}} \right) \le \frac{\alpha}{12}.
  \end{align}

  \noindent If we group together all the eigenvalues with
  $\lambda_i \in [1/2, 1]$ then using Theorem
  \ref{thm:MassartConcentrationLemma} we have that
  \begin{align} \label{eq:proof:lemma:large mass:mediumBigEigenvalues}
    \Prob_{\vec{x} \sim \normal(\vec{0}, \matr{I})} \left( \sum_{i : \lambda_i \in [\frac{1}{2}, 1]} \left(\frac{1}{\lambda_i} - 1\right) (x_i^2 - 1) \ge 2 \sqrt{\sum_{i : \lambda_i \in [\frac{1}{2}, 1]} \left(\frac{1}{\lambda_i} - 1\right)^2 \log\frac{12}{\alpha}} + 4 \log\frac{12}{\alpha} \right) \le \frac{\alpha}{12}.
  \end{align}

  \noindent Now we combine
  \eqref{eq:proof:lemma:large mass:mediumSmallEigenvalues} with
  \eqref{eq:proof:lemma:large mass:mediumBigEigenvalues} and using the
  first assumption of the lemma, namely that
  $\sum_i (1/\lambda_i - 1)^2 < B^2$, we get that
  \begin{align*}
    \Prob_{\vec{x} \sim \normal(\vec{0}, \matr{I})} \left( \sum_{i : \lambda_i \in [\frac{1}{2}, 2]} \left(\frac{1}{\lambda_i} - 1\right) x_i^2 \ge
    \sum_{i : \lambda_i \in [\frac{1}{2}, 2]} \left(\frac{1}{\lambda_i} - 1\right) + (4 + 2 \cdot B) \log\frac{12}{\alpha} \right) \le \frac{\alpha}{6}.
  \end{align*}
  \noindent Moreover, for all $\lambda_i \in [1/2, 2]$, it holds that
  $(1/\lambda_i - 1) + \log \lambda_i \le 4 (\lambda_i - 1)^2$ which combined
  with the above inequality gives us
  \begin{align*}
    \Prob_{\vec{x} \sim \normal(\vec{0}, \matr{I})} \left( \sum_{i : \lambda_i \in [\frac{1}{2}, 2]} \left[ \left(\frac{1}{\lambda_i} - 1\right) x_i^2 + \log \lambda_i \right] \ge
    \sum_{i : \lambda_i \in [\frac{1}{2}, 2]} 4 \left(\lambda_i^2 - 1\right) + (4 + 2 \cdot B) \log\frac{12}{\alpha} \right) \le \frac{\alpha}{6}.
  \end{align*}
  where if we apply \eqref{eq:proof:lemma:large mass:firstAssumption} then we
  get that
  \begin{align} \label{eq:proof:lemma:large mass:mediumEigenvalues}
    \Prob_{\vec{x} \sim \normal(\vec{0}, \matr{I})} \left( \sum_{i : \lambda_i \in [\frac{1}{2}, 2]} \left[ \left(\frac{1}{\lambda_i} - 1\right) x_i^2 + \log \lambda_i \right] \ge
    4 B^4 + (4 + 2 \cdot B) \log\frac{12}{\alpha} \right) \le \frac{\alpha}{6}.
  \end{align}

  \noindent We now combine everything together. Using (a), (b) and (c) we have
  that
  \begin{align*}
    \sum_i \left[\frac{(x_i - \mu_i)^2}{\lambda_i} - x_i^2 + \log \lambda_i\right] & = \sum_{i : \lambda_i < 1/2} B x_i^2 - 2 \sum_i \frac{x_i \mu_i}{\lambda_i} + \sum_{i : \lambda_i \in [\frac{1}{2}, 2]} \left[ \left(\frac{1}{\lambda_i} - 1\right) x_i^2 + \log \lambda_i \right] \\
    & ~~~~~~~~~~~~~~ + \sum_{i : \lambda_i > 2} \log B.
  \end{align*}
  If in the above expression we use the equations
  \eqref{eq:proof:lemma:large mass:smallEigenvalues},
  \eqref{eq:proof:lemma:large mass:outsideEigenvalues}, and
  \eqref{eq:proof:lemma:large mass:mediumEigenvalues} together with the fact
  that the number of $\lambda_i$'s that have $\lambda_i > 2$ is at most $4 B^4$,
  as we noted below the statement (a) - (c), we get the following
  \begin{align} \label{eq:proof:lemma:large mass:combinedEigenvalues}
    \Prob_{\vec{x} \sim \normal(\vec{0}, \matr{I})} \left( \sum_i \left[\frac{(x_i - \mu_i)^2}{\lambda_i} - x_i^2 + \log \lambda_i\right] \ge
    23 \cdot B^5 \cdot \log\left( \frac{12}{\alpha} \right) \right) \le \frac{\alpha}{2}.
  \end{align}
  \noindent Finally the lemma follows by combining the equations
  \eqref{eq:proof:lemma:large mass:massRatioLowerBound},
  \eqref{eq:proof:lemma:large mass:massRatioExpansion}, and
  \eqref{eq:proof:lemma:large mass:combinedEigenvalues}.
\end{proof}

  Our next goal is to apply Lemma \ref{lemma:large mass} to bound the measure
assigned to $S$ by $\normal(\hat{\vec{\mu}}_S, \hat{\matr{\Sigma}}_S)$. For
this, we need to convert the bounds given by Corollary
\ref{corollary:initialization} to those required to apply Lemma
\ref{lemma:large mass}.

\begin{proposition} \label{prop:massOfEmpirical}
    It holds that
  \begin{enumerate} \label{prop:bounds}
    \item[$\triangleright$]
      $\norm{\matr I - \matr \Sigma^{* 1/2} \matr{\hat{\Sigma}_S}^{-1} \matr \Sigma^{* 1/2}}_F \le O\left(\frac{\log(1/\alpha)}{\alpha^2}\right)
      \quad \mathrm{ and } \quad
      \norm{\matr I - \matr{\hat{\Sigma}_S}^{1/2} \matr \Sigma^{* -1} \matr{\hat{\Sigma}_S}^{1/2}}_F \le O\left(\frac{\log(1/\alpha)}{\alpha^2}\right)$,
    \item[$\triangleright$]
      $\Omega(\alpha^2) \cdot \matr{I} \preceq \matr \Sigma^{* -1/2} \matr{\hat{\Sigma}_S} \matr \Sigma^{* -1/2} \le O\left(\frac{1}{\alpha^2}\right) \cdot \matr{I}
      \quad \mathrm{ and } \quad
      \Omega(\alpha^2) \cdot \matr{I} \preceq \matr{\hat{\Sigma}_S}^{-1/2} \matr \Sigma^* \matr{\hat{\Sigma}_S}^{-1/2} \preceq O\left(\frac{1}{\alpha^2}\right) \cdot \matr{I}$,
    \item[$\triangleright$]
      $\norm{\matr{\hat{\Sigma}_S}^{-1} \matr \Sigma^{* 1/2} (\vec{\hat{\mu}}_S - \vec{\mu}^*)}_2 \le O\left(\frac{\log(1/\alpha)}{\alpha^2}\right)
      \quad \mathrm{ and } \quad
      \norm{\matr \Sigma^{* -1} \matr{\hat{\Sigma}_S}^{1/2} (\vec{\hat{\mu}}_S - \vec{\mu}^*)}_2 \le O\left(\frac{\log(1/\alpha)}{\alpha^2}\right)$.
  \end{enumerate}
\end{proposition}

\begin{proof}
    Using the same argument as in the beginning of the proof of Lemma
  \ref{lem:conditional mean closeness} we may assume without loss of
  generality that $(\vec{\mu}^*, \matr{\Sigma}^*) = (\vec 0, \matr I)$ and that
  $\matr{\hat{\Sigma}}_S = \diag(\lambda_1, \dots, \lambda_d)$. Also, for
  simplicity we use $\vec{\mu}$ to denote $\vec{\hat{\mu}}_S$.
  \medskip

  \noindent From parts 1. and 2. of Corollary \ref{corollary:initialization} we
  have that
  \[ \norm{\matr I - \matr \Sigma^{* 1/2} \matr{\hat{\Sigma}_S}^{-1} \matr \Sigma^{* 1/2}}_F =
   \sum_{i} \left(1 - \frac{1}{\lambda_i}\right)^2 \le \frac{1}{\alpha^2} \sum_{i} (1 - \lambda_i)^2 \le O\left(\frac{\log(1/\alpha)}{\alpha^2}\right). \]

  \noindent From parts 1. and 2. of Corollary \ref{corollary:initialization}
  and the fact that the Frobenius norm is an upper bound to the spectral norm
  we have that
  \begin{align} \label{eq:proof:prop:massOfEmpirical:1}
    \Omega(\alpha^2) \cdot \matr{I} \preceq \matr{\Sigma}^{* -1/2} \matr{\hat{\Sigma}_S} \matr \Sigma^{* -1/2} \preceq O\left(\log\frac{1}{\alpha}\right) \cdot \matr{I} \preceq O\left(\frac{1}{\alpha^2}\right) \cdot \matr{I}.
  \end{align}
  \noindent From parts 2. and 3. of Corollary \ref{corollary:initialization} we
  have that
  \[ \norm{\matr{\hat{\Sigma}_S}^{-1} \matr \Sigma^{* 1/2} (\vec{\hat{\mu}}_S - \vec{\mu}^*)}_2 = \sum_{i} \frac{1}{\lambda_i^2}  \mu_i^2 \le O\left(\frac{\log(1/\alpha)}{\alpha^2}\right). \]

  \noindent Similarly from parts 1. and 2. of Corollary
  \ref{corollary:initialization} we have that
  \[ \norm{\matr I - \matr{\hat{\Sigma}_S}^{* 1/2} \matr \Sigma^{* -1} \matr{\hat{\Sigma}_S}^{* 1/2}}_F
   =
   \sum_{i} (1-\lambda_i)^2 \le O\left({\log(1/\alpha)}\right). \]

  \noindent Also \eqref{eq:proof:prop:massOfEmpirical:1} directly implies that
  \[ \Omega(\alpha^2) \cdot \matr{I} \preceq \matr{\hat{\Sigma}_S}^{-1/2} \matr \Sigma^* \matr{\hat{\Sigma}_S}^{-1/2} \preceq O\left(\frac{1}{\alpha^2}\right) \cdot \matr{I}. \]

  \noindent Finally from parts 2. and 3. of Corollary
  \ref{corollary:initialization}, we have that
  \[ \norm{\matr{\hat{\Sigma}_S}^{-1} \matr \Sigma^{* 1/2} (\vec \mu_1 - \vec \mu_2)}_2 = \sum_{i} {\lambda_i^2} \mu_i^2 \le O\left(\frac{\log(1/\alpha)}{\alpha^2}\right). \]
\end{proof}

\noindent Proposition \ref{prop:bounds} implies that Lemma
\ref{lemma:large mass} can be invoked with
$B = O\left(\frac{\log(1/\alpha)}{\alpha^2} \right)$ to obtain the following.

\begin{corollary}\label{corollary:conditional mass}
    Consider a truncated normal distribution
  $\normal(\vec{\mu}^*, \matr{\Sigma}^*, S)$ with
  $\normal(\vec{\mu}^*, \matr{\Sigma}^*; S) \ge \alpha > 0$. The estimates
  $(\hat{\vec{\mu}}_S, \hat{\matr{\Sigma}}_S)$ obtained by Corollary
  \ref{corollary:initialization}, satisfy
  $\normal(\hat{\vec{\mu}}_S, \hat{\matr{\Sigma}}_S; S) \ge c_\alpha$ for some
  constant $c_\alpha$ that depends only on the constant $\alpha$.
\end{corollary}

\begin{corollary}\label{corollary:box mass}
    Consider a truncated normal distribution
  $\normal(\vec{\mu}^*, \matr{\Sigma}^*, S)$ with
  $\normal(\vec{\mu}^*, \matr{\Sigma}^*; S) \ge \alpha > 0$. Let
  $(\hat{\vec{\mu}}_S, \hat{\matr{\Sigma}}_S)$ be the estimate obtained by
  Corollary \ref{corollary:initialization} and let
  $({\vec{\mu}}, {\matr{\Sigma}})$ be any estimate that satisfies
  \begin{itemize}
    \item[$\triangleright$]
      $\norm{\matr I - \matr{\hat{\Sigma}_S}^{1/2} \matr \Sigma^{-1} \matr{\hat{\Sigma}_S}^{1/2}}_F \le O\left(\frac{\log(1/\alpha)}{\alpha^2}\right)$,
    \item[$\triangleright$]
      $\Omega(\alpha^2) \cdot \matr{I} \preceq \matr{\hat{\Sigma}_S}^{1/2} \matr \Sigma^{-1} \matr{\hat{\Sigma}_S}^{1/2} \preceq O\left(\frac{1}{\alpha^2}\right) \cdot \matr{I}$,
    \item[$\triangleright$]
      $\norm{\matr \Sigma^{-1} \matr{\hat{\Sigma}_S}^{1/2} (\vec{\hat{\mu}}_S - \vec{\mu})}_2 \le O(\frac{\log(1/\alpha)}{\alpha^2})$.
  \end{itemize}
  Then, $\normal({\vec{\mu}}, {\matr{\Sigma}}; S) \ge c_\alpha$ for some
  constant $c_\alpha$ that depends only on the constant $\alpha$.
\end{corollary}

\subsection{Analysis of Stochastic Gradient Descent -- Proof of Theorem \ref{thm:main theorem}} \label{sec:sgd-final}

  As we explained in the introduction, our estimation algorithm is projected
stochastic gradient descent for maximizing the population log-likelihood
function with the careful choice of the projection set. Because the proof of
consistency and efficiency of our algorithm is technical we first present an
outline of the proof and then we present the individual lemmas for each step.

The framework that we use for our analysis is based on the Chapter 14 of
\cite{ShalevS14}. In this framework the goal is to minimize a convex function
$f : \reals^k \to \reals$ by doing gradient steps but with noisy unbiased
estimations of the gradient. For us $k$ is equal to $d^2 + d$ since we need to
estimate the entries of the mean vector and the covariance matrix. The
following algorithm describes projected stochastic gradient descent applied to
a convex function $f$, with projection set $\Domain_r \subseteq \reals^k$. We
will define formally the particular set $\Domain_r$ that we consider in
Definition \ref{def:setaki}. For now $\Domain_r$ should be thought of an
arbitrary convex subset of $\reals^k$.

\begin{algorithm}[H]
\caption*{\textbf{Algorithm ($\star$). Projected SGD for Minimizing a $\boldsymbol{\lambda}$-Strongly Convex Function.}}
\begin{algorithmic}[1]
\State $\vec{w}^{(0)} \gets$ arbitrary point in $\Domain_r$ \Comment{(a) \textit{initial feasible point}}
\For{$i = 1, \dots, M$}
  \State Sample $\vec{v}^{(i)}$ such that $\Exp\left[ \vec{v}^{(i)} \mid \vec{w}^{(i - 1)}\right] \in \partial f(\vec{w}^{(i - 1)})$ \Comment{(b) \textit{estimation of gradient}}
  \State $\vec{r}^{(i)} \gets \vec{w}^{(i - 1)} - \frac{1}{\lambda \cdot i} \vec{v}^{(i)}$   \State $\vec{w}^{(i)} \gets \argmin_{\vec{w} \in \Domain_r} \norm{\vec{w} - \vec{r}^{(i)}}$ \Comment{(c) \textit{projection step}}
\EndFor
\State \textbf{return} $\bar{\vec{w}} \gets \frac{1}{M} \sum_{i = 1}^M \vec{w}^{(i)}$
\end{algorithmic}
\label{alg:projectedSGDGeneral}
\end{algorithm}

  Our goal is to apply the Algorithm
\hyperref[alg:projectedSGDGeneral]{($\star$)} to the negative log-likelihood
function that we defined in \eqref{eq:logLikelihoodPopulation}. In order to
apply Algorithm \hyperref[alg:projectedSGDGeneral]{($\star$)} we have to first
solve the following three algorithmic problems
\begin{enumerate}
  \item[(a)] \textbf{initial feasible point:} efficiently compute an initial
    feasible point in $\Domain_r$,
  \item[(b)] \textbf{unbiased gradient estimation:} efficiently sample an
    unbiased estimation of $\nabla f$,
  \item[(c)] \textbf{efficient projection:} design an efficient algorithm to
    project to the set $\Domain_r$.
\end{enumerate}
Then our goal is to apply the following theorem of \cite{ShalevS14}.

\begin{theorem}[Theorem 14.11 of \cite{ShalevS14}.] \label{thm:mainSGDAnalysis}
    Let $f : \reals^k \to \reals$ be a convex function and
  $\vec{w}^{(1)}, \dots, \vec{w}^{(M)}$ be the sequence produced by Algorithm
  \hyperref[alg:projectedSGDGeneral]{($\star$)}, where
  $\vec{v}^{(1)}, \dots, \vec{v}^{(M)}$ is a sequence of random vectors such
  that
  $\Exp\left[ \vec{v}^{(t)} \mid \vec{w}^{(t - 1)}\right] = \nabla f(\vec{w}^{(t - 1)})$
  for all $t \in [M]$, and let
  $\vec{w}^* = \arg \min_{\vec{w} \in \Domain_r} f(\vec{w})$ be a minimizer
  of $f$. If we assume the following
\begin{enumerate}
  \item[\emph{(i)}] \emph{\textbf{bounded variance step:}}
    $\Exp\left[ \norm{\vec{v}^{(t)}}_2^2 \right] \le \rho^2$,
  \item[\emph{(ii)}] \emph{\textbf{strong convexity:}} for all $i \in [n]$,
  the function $f_i(\cdot)$ is $\lambda$-strongly convex,
\end{enumerate}
\noindent then,
$\Exp\left[ f(\vec{\bar{w}}) \right] - f(\vec{w}^*) \le \frac{\rho^2}{2 \lambda M}\left( 1 + \log(M) \right),$
where $\bar{\vec{w}}$ is the output of the Algorithm
\hyperref[alg:projectedSGDGeneral]{($\star$)}.
\end{theorem}

\noindent Unfortunately, none of the properties (i), (ii) hold for all
vectors $\vec{w} \in \reals^d$ and hence we cannot use vanilla SGD. For this
reason, we add the projection step. We identify a projection set
$\Domain_{r^*}$ such that log-likelihood satisfies both (i) and (ii) for all
vectors $\vec{w} \in \Domain_{r^*}$. We to describe the domain $\Domain_{r^*}$
of our projected stochastic gradient ascent in the transformed space where we
have applied an affine transformation such that $\hat{\vec{\mu}}_S = \vec{0}$
and $\hat{\matr{\Sigma}}_S = \matr{I}$. The domain is parameterized by the
positive number $r$ and is defined as follows

\begin{definition}[\textsc{Projection Set}] \label{def:setaki}
We define
\begin{equation} \label{eq:domainDefinition}
\Domain_r = \left\{ (\vec{\nu}, \matr{T}) ~\left|~
                     \norm{\vec{\nu}}_2 \le r,~
                     \norm{\matr{I} - \matr{T}}_F \le r,~
                     \norm{\matr{T}^{-1}}_2 \le r
                  \right. \right\}.
\end{equation}
We set $r^* = O(\log(1/\alpha)/\alpha^2)$. Given $\vec{x} \in \reals^k$ we say
that $\vec{x}$ is \textbf{feasible} if and only if $\vec{x} \in \Domain_{r^*}$.
\end{definition}

\noindent Observe that $\Domain_r$ is a convex set as the constraint
$\norm{\matr{T}^{-1}}_2 \le r$ an infinite set of linear, with respect to
$\matr{T}$, constraints of the form $\vec{x}^T \matr{T} \matr{x} \ge 1/r$ for
all $\vec{x} \in \reals^d$. Moreover, in Algorithm \ref{alg:projectToDomain} we
present an efficient procedure to project any point in our space to $\Domain_r$.
Using the projection to $\Domain_{r^*}$, we can prove (i) and (ii) and hence we
can apply Theorem \ref{thm:mainSGDAnalysis}. The last step is to transform the
conclusions of Theorem \ref{thm:mainSGDAnalysis} to guarantees in the parameter
space. For this we use again the strong convexity of $f$ which implies that
closeness in the objective value translates to closeness in the parameter space.
For this argument we also need the following property:
\begin{enumerate}
  \item[(iii)] \textbf{feasibility of optimal solution:}
    $\vec{w}^* \in \Domain_{r^*}$.
\end{enumerate}

\paragr{Outline of the Proof of Theorem \ref{thm:main theorem}.}
The proof proceeds by solving the aforementioned algorithmic problems (a) - (c)
and the statistical problems (i) - (iii).
\begin{Enumerate}
  \item In Section \ref{sec:sgd:initialization} we describe our initialization
    step which solves problem (a).
  \item We start in Section \ref{sec:sgd:gradient} we describe the rejection
    sampling algorithm to get an unbiased estimate of the gradient which solves
    the problem (b).
  \item In Section \ref{sec:sgd:projection} we present a detailed analysis of
    our projection onto the set $\Domain_r$ algorithm which gives a solution to
    algorithmic problem (c).
  \item In Section \ref{sec:sgd:statistical} we present a proof of the (i)
    bounded variance property. For the property (ii) strong convexity we use
    Lemma \ref{lem:strongConcavityofLogLikelihoodWithTruncation} from Section
    \ref{sec:strong-convexity}.
  \item The feasibility of the optimal solution follows directly from Corollary
    \ref{corollary:initialization} which resolves the problem (iii).
  \item Finally in Section \ref{sec:sgd:mainResultProof} we use all the
    mentioned results to prove our main Theorem \ref{thm:main theorem}.
\end{Enumerate}
\medskip

\noindent We define $\vec{t} = \matr{T}^{\flat}$ and
$\vec{w} = \begin{bmatrix} \vec{t} \\ \vec{v} \end{bmatrix}$ for simplicity.
Our algorithm iterates over the estimation $\vec{w}$ of the true parameters
$\vec{w}^*$. Let also $\samplO$ the sample oracle from the unknown distribution
$\normal(\vec{\mu}^*, \matr{\Sigma}^*, S)$.

\subsubsection{Initialization Step} \label{sec:sgd:initialization}

  The initialization step of our algorithm computes the empirical mean
$\hat{\vec{\mu}}$ and the empirical covariance matrix $\hat{\matr{\Sigma}}$
of the truncated distribution $\normal(\vec{\mu}^*, \matr{\Sigma}^*, S)$ using
$n = \tilde{O}\left( \frac{d^2}{\eps^2} \right)$ samples
$\vec{x}_1, \dots, \vec{x}_n$.
\begin{equation} \label{eq:empiricalMeanAndCovarianceEstimation}
\hat{\vec{\mu}}_S = \frac{1}{n} \sum_{i = 1}^n \vec{x}_i, ~~~~~~~~~
\hat{\matr{\Sigma}}_S = \frac{1}{n} \sum_{i = 1}^n \left( \vec{x}_i -
\hat{\vec{\mu}}_S \right) \left( \vec{x}_i - \hat{\vec{\mu}}_S \right)^T.
\end{equation}
\noindent Then, we apply the affine transformation in our basis so that
$\hat{\vec{\mu}}_S = \vec{0}$ and $\hat{\matr{\Sigma}}_S = \matr{I}$. Then we
set our initial estimate
$\vec{w}^{(0)} = \begin{bmatrix} \matr{I}^{\flat} \\ \vec{0} \end{bmatrix}$.

\subsubsection{Unbiased Estimation of the Gradient} \label{sec:sgd:gradient}

  From Section \ref{sec:strong-convexity} given some parameters
$(\vec{\nu}, \matr{T})$ the expected gradient of the population log-likelihood
function is equal to
\[ - \Exp_{\vec{x} \sim \normal(\vec{\mu^*}, \matr{\Sigma^*}, S)}
  \left[
    \begin{bmatrix}
      \left( - \frac{1}{2} \vec{x} \vec{x}^T \right)^{\flat} \\
      \vec{x}
    \end{bmatrix}
  \right] +
  \Exp_{\vec{z} \sim \normal(\matr{T}^{-1} \vec{\nu}, \matr{T}^{-1}, S)}
  \left[
    \begin{bmatrix}
      \left( - \frac{1}{2} \vec{z} \vec{z}^T \right)^{\flat} \\
      \vec{z}
    \end{bmatrix}
  \right]. \]
\noindent To compute an unbiased estimate of the first term we use one of the
samples provided by the oracle $\samplO$. To compute an unbiased estimate of
the second term we sample a vector $\vec{z}$ from
$\normal(\matr{T}^{-1} \vec{\nu}, \matr{T}^{-1})$ and we check if
$\vec{z} \in S$ using $\memb_S$. If $\vec{z}$ is in $S$ then we use it to get
an unbiased estimate of the second term of the gradient, otherwise we repeat
the rejection sampling procedure until we succeed. From the definition of the
projection set $\Domain_r$ and from Corollary \ref{corollary:box mass} we have
that this rejection sampling procedure takes only $O(1/c_{\alpha})$ samples,
with high probability, where $c_{\alpha}$ is the constant from Corollary
\ref{corollary:box mass}.

\subsubsection{Projection Algorithm} \label{sec:sgd:projection}

  The next Lemma \ref{lem:projectionProcedure} shows gives the missing details
and proves the correctness of Algorithm \ref{alg:projectToDomain}.

\begin{lemma} \label{lem:projectionProcedure}\label{lemma:efficient projection}
  Given $(\vec{\nu}', \matr{T}')$, there exists an efficient algorithm that
solves the following problem which corresponds to projecting
$(\vec{\nu}, \matr{T})$ to the set $\Domain_r$
\[ \argmin_{(\vec{\nu}, \matr{T}) \in \Domain_r
} \norm{\vec{\nu} - \vec{\nu}'}_2^2 + \norm{\matr{T} - \matr{T}'}_F^2. \]
\end{lemma}

\begin{proof}
    Because of the form of $\Domain_r$ and the objective function of our
  program, observe that we can project $\vec{\nu}$ and $\matr{T}$ separetely.
  The projection of $\vec{\nu}'$ to $\Domain_r$ is the solution of the following
  problem $\arg \min_{\vec{\nu} : \norm{\vec{\nu}}_2 \le r}
  \norm{\vec{\nu} - \vec{\nu}'}_2^2$ which has a closed form. So we focus on
  the projection of $\matr{T}$.

  \noindent To project $\matr{T}$ to $\Domain_r$ we need to solve the following
  program.
  \begin{align*}
    \argmin_{\matr{T}} & \norm{\matr{T} - \matr{T}'}_F^2 \\
  \text{s.t.} ~~~        & \norm{\matr{T} - \matr{I}}_F^2 \le r^2 \\
                         & \matr{T} \succeq \frac{1}{r} \matr{I}
  \end{align*}
  \noindent Equivalently, we can perform binary search over the Lagrange multiplier
  $\lambda$ and at each step solve the following program.
  \begin{align*}
    \arg \min_{\matr{T}} & \norm{\matr{T} - \matr{T}'}_F^2 + \lambda \norm{\matr{T} - \matr{I}}_F^2 \\
  \text{s.t.} ~~~        & \matr{T} \succeq \frac{1}{r} \matr{I}
  \end{align*}
  \noindent After completing the squares on the objective, we can set
  $\matr{H} = \matr{I} - \frac{1}{1 + \lambda} \left( \matr{T}' + \lambda \matr{I} \right)$
  and make the change of variables $\matr{R} = \matr{I} - \matr{T}$ and solve
  the program.
  \begin{align*}
    \arg \min_{\matr{T}} & \norm{\matr{R} - \matr{H}}_F^2 \\
  \text{s.t.} ~~~        & \matr{R} \preceq \left( 1 - \frac{1}{r} \right) \matr{I}
  \end{align*}
  \noindent Observe now that without loss of generality $\matr{H}$ is diagonal.
  If this is not the case we can compute the singular value decomposition of
  $\matr{H}$ and change the base of the space so that $\matr{H}$ is diagonal.
  Then, after finding the answer to this case we transform back the space to
  get the correct $\matr{R}$. When $\matr{H}$ is diagonal the solution to this
  problem is very easy and it even has a closed form.
\end{proof}

\subsubsection{Bound on the Variance of Gradient Estimation} \label{sec:sgd:statistical}

\noindent As we explained before, apart from efficient projection
(Lemma \ref{lem:projectionProcedure}) and strong convexity
(Lemma \ref{lem:strongConcavityofLogLikelihoodWithTruncation}) we also need a
bound on the square of the norm of the gradient vector in order to prove
theoretical guarantees for our SGD algorithm.

\begin{lemma} \label{lem:normOfStochasticGradient}\label{lemma:unbiased estimate of sample}
    Let $\vec{v}^{(i)}$ the gradient of the log likelihood function at step
  $i$ as computed in the line 6 of Algorithm \ref{alg:projectedSGD}. Let
  $\vec{\nu}$, $\matr{T}$ be the guesses of the parameters after step $i - 1$
  according to which the gradient is computed with
  $\vec{\mu} = \matr{T}^{-1} \vec{\nu}$ and
  $\matr{\Sigma} = \matr{T}^{-1}$. Let $\vec{\mu}^*$, $\matr{\Sigma}^*$ be
  the parameters we want to recover, with $\vec{\nu}^* = \matr{\Sigma}^{* -1}
  \vec{\mu}^*$ and $\matr{T}^* = \matr{\Sigma}^{* -1}$. Define We assume that
  $(\vec{\nu}, \matr{T}) \in \Domain_r$,
  $(\vec{\nu}^*, \matr{T}^*) \in \Domain_r$ and that
  $\normal(\vec{\mu}, \matr{\Sigma}; S) \ge \beta$,
  $\normal(\vec{\mu}^*, \matr{\Sigma}^*; S) \ge \beta$. Then, we have that
  \[ \Exp\left[ \norm{\vec{v}^{(i)}}_2^2 \right] \le \frac{100}{\beta} d^2 r^2. \]
\end{lemma}

\begin{proof}
  Let $\vec{\mu} = \matr{T}^{-1} \vec{\nu}$ and
$\matr{\Sigma} = \matr{T}^{-1}$. According to Algorithm
\ref{alg:estimateGradient} and equation
\eqref{eq:gradientofLikelihoodOneSample} we have that
\begin{align*}
  \Exp\left[ \norm{\vec{v}^{(i)}}_2^2 \right] & =
  \Exp_{\vec{x} \sim \normal(\vec{\mu}^*, \matr{\Sigma}^*, S)}\left[
    \Exp_{\vec{y} \sim \normal(\vec{\mu}, \matr{\Sigma}, S)}\left[
      \norm{
        \begin{bmatrix} \left( - \frac{1}{2} \vec{x} \vec{x}^T \right)^{\flat} \\
                        \vec{x} \end{bmatrix}
     -  \begin{bmatrix} \left( - \frac{1}{2} \vec{y} \vec{y}^T \right)^{\flat} \\
                        \vec{y} \end{bmatrix}}_2^2
    \right]
  \right] \\
      & \le 3 \Exp_{\vec{x} \sim \normal(\vec{\mu}^*, \matr{\Sigma}^*, S)}\left[
          \norm{\begin{bmatrix} \left( - \frac{1}{2} \vec{x} \vec{x}^T \right)^{\flat} \\ \vec{x} \end{bmatrix}
          }_2^2
      \right] +
         3 \Exp_{\vec{y} \sim \normal(\vec{\mu}, \matr{\Sigma}, S)}\left[
           \norm{\begin{bmatrix} \left( - \frac{1}{2} \vec{y} \vec{y}^T \right)^{\flat} \\ \vec{y} \end{bmatrix}
           }_2^2
         \right].
\end{align*}
\noindent In order to bound each of these terms we use two facts: (1) that
the parameters $(\vec{\mu}^*, \matr{\Sigma}^*)$, $(\vec{\mu}, \matr{\Sigma})$
belong in $\Domain_r$, (2) the measure of $S$ is greater than $\beta$ for both
sets of parameters, i.e. $\normal(\vec{\mu}, \matr{\Sigma}; S) \ge \beta$ and
$\normal(\vec{\mu}^*, \matr{\Sigma}^*; S) \ge \beta$. Hence, we will show how
to get an upper bound for the second term and the upper bound of the first term
follows the same way.

\begin{align*}
  \Exp_{\vec{x} \sim \normal(\vec{\mu}, \matr{\Sigma}, S)}\left[
      \norm{\begin{bmatrix} \left( - \frac{1}{2} \vec{x} \vec{x}^T \right)^{\flat} \\ \vec{x} \end{bmatrix}
      }_2^2
  \right] & = \frac{1}{2} \Exp_{\vec{x} \sim \normal(\vec{\mu}, \matr{\Sigma}, S)}\left[
      \norm{\vec{x} \vec{x}^T}_F^2
  \right] + \Exp_{\vec{x} \sim \normal(\vec{\mu}, \matr{\Sigma}, S)}\left[ \norm{\vec{x}}_2^2 \right] \\
          & \le \frac{1}{2} \frac{1}{\beta} \Exp_{\vec{x} \sim \normal(\vec{\mu}, \matr{\Sigma})}\left[
              \norm{\vec{x} \vec{x}^T}_F^2
          \right] + \frac{1}{\beta} \Exp_{\vec{x} \sim \normal(\vec{\mu}, \matr{\Sigma})}\left[ \norm{\vec{x}}_2^2 \right]
\end{align*}

\noindent We define $\vec{\rho} = \begin{bmatrix} \sigma_{1 1} & \sigma_{2 2} & \cdots & \sigma_{d d} \end{bmatrix}^T$,
by straightforward calculations of the last two expressions we get that

\begin{align*}
  \Exp_{\vec{x} \sim \normal(\vec{\mu}, \matr{\Sigma}, S)}\left[
      \norm{\begin{bmatrix} \left( - \frac{1}{2} \vec{x} \vec{x}^T \right)^{\flat} \\ \vec{x} \end{bmatrix}
      }_2^2
  \right] & \le \frac{2}{\beta} \left(\norm{\matr{\Sigma}}_F^2
                + \norm{\vec{\rho}}_2^4 + \norm{\vec{\rho}}_2^2 \norm{\vec{\mu}}_2^2 +
                2 \vec{\mu}^T \matr{\Sigma} \vec{\mu} + \norm{\vec{\mu}}_2^4 + \norm{\vec{\rho}}_2^2 +
                \norm{\vec{\mu}}_2^2\right)
\end{align*}
\noindent But from the fact that $(\vec{\nu}, \matr{T}) \in \Domain_r$ we can
get the following bounds
\[ \norm{\matr{\Sigma}}_F^2 \le d \cdot r,
   ~~~ \norm{\vec{\rho}}_2^2 \le d \cdot r,
   ~~~ \vec{\mu}^T \matr{\Sigma} \vec{\mu} \le r,
   ~~~ \norm{\vec{\mu}}_2^2 \le r
\]

\noindent From which we get that
\begin{align*}
  \Exp_{\vec{x} \sim \normal(\vec{\mu}, \matr{\Sigma}, S)}\left[
      \norm{\begin{bmatrix} \left( - \frac{1}{2} \vec{x} \vec{x}^T \right)^{\flat} \\ \vec{x} \end{bmatrix}
      }_2^2
  \right] & \le \frac{16}{\beta} \cdot d^2 \cdot r^2
\end{align*}
\noindent Finallly, we apply these bounds to the first bound for
$\Exp\left[ \norm{\vec{v}^{(i)}}_2^2 \right]$ and the lemma follows.
\end{proof}

\subsubsection{Proof of Theorem \ref{thm:main theorem}} \label{sec:sgd:mainResultProof}

\noindent Now we have all the ingredients to use the basic tools for analysis
of projected stochastic gradient descent. As we already explained, the
formulation we use is from Chapter 14 of \cite{ShalevS14}.

\noindent Before stating the proof we present a simple lemma for strongly convex
functions that follows easily from the definition of strong convexity.

\begin{lemma}[Lemma 13.5 of \cite{ShalevS14}.] \label{lem:strongConvexityToDistance}
  If $f$ is $\lambda$-strongly convex and $\vec{w}^*$ is a minimizer of $f$,
then, for any $\vec{w}$ it holds that
\[ f(\vec{w}) - f(\vec{w}^*) \ge \frac{\lambda}{2} \norm{\vec{w} - \vec{w}^*}_2^2. \]
\end{lemma}

\noindent Using Theorem \ref{thm:mainSGDAnalysis}, together with Lemmata
\ref{lem:strongConcavityofLogLikelihoodWithTruncation},
\ref{lem:normOfStochasticGradient} and \ref{lem:strongConvexityToDistance}
we can get our first theorem that bounds the expected cost of Algorithm
\ref{alg:projectedSGD}. Then we can also use Markov's inequality to get
and our first result in probability.

\begin{lemma} \label{lem:expectedPerformanceSGD}
    Let $\vec{\mu}^*$, $\matr{\Sigma}^*$ be the underline parameters of our
  model, $f = - \bar{\ell}$,
  $r^* = O\left( \frac{\log(1/\alpha)}{\alpha^2} \right)$ and also
  \[ \beta_* = \min_{(\vec{\nu}, \matr{T}) \in \Domain_r} \normal(\matr{T}^{-1}\vec{\nu}, \matr{T}^{-1}; S) ~~~ \text{ and } ~~~
     \lambda_* = \min_{(\vec{\nu}, \matr{T}) \in \Domain_r} \lambda_m(\matr{T}^{-1} \vec{\nu}, \matr{T}^{-1}) \ge h(r^*). \]
  then there exists a universal constant $C > 0$ such that
  \[ \Exp\left[ f(\vec{\bar{w}}) \right] - f(\vec{w}^*) \le \frac{C \cdot h(r^*)}{\beta_*^5} \cdot \frac{d^2}{M} \left( 1 + \log(M) \right), \]
  \noindent where $\vec{\bar{\vec{w}}}$ is the output of Algorithm
  \ref{alg:projectedSGD}, and
  $h(x) \triangleq \frac{x}{4 (1 + x)^2 x^2 + \sqrt{1 + x}}$.
\end{lemma}

\begin{proof}
    This result follows directly from Theorem \ref{thm:mainSGDAnalysis},
  if our initial estimate $\vec{w}^{(0)}$ belongs to $\Domain_{r^*}$. To ensure
  this we set
  $r^* = O\left( \frac{\log \left(1 / \alpha\right)}{\alpha^2}\right)$ and we
  apply Proposition \ref{prop:bounds} and the Lemma follows.
\end{proof}

\noindent Using Lemma \ref{lem:expectedPerformanceSGD} and applying Markov's
inequality we get that
\begin{equation} \label{eq:MarkovFinalBound}
  \Prob \left( f(\vec{\bar{w}}) - f(\vec{w}^*) \ge \frac{3 C \cdot h(r^*)}{\beta_*^5} \cdot \frac{d^2}{M}
\left( 1 + \log(M) \right) \right) \le \frac{1}{3}.
\end{equation}
We can easily amplify the probability of success to $1 - \delta$ by repeating
$\log(1/\delta)$ from scratch the optimization procedure and keeping the
estimation that achieves the maximum log-likelihood value. The high probability
result enables the use of Lemma \ref{lem:strongConvexityToDistance} to get
closeness in parameter space.

To get our estimation we first repeat the SGD procedure $K = \log(1/\delta)$
times independently, with parameter $M$ each time. We then get the set of
estimates $\mathcal{E} = \{\bar{\vec{w}}_1, \bar{\vec{w}}_2, \dots,
\bar{\vec{w}}_K\}$. Our ideal final estimate would be
\[ \hat{\vec{w}} = \arg \min_{\bar{\vec{w}} \in \mathcal{E}} \bar{\ell}(\vec{w}) \]
\noindent but we don't have access to the exact value of $\bar{\ell}(\vec{w})$.
Because of \eqref{eq:MarkovFinalBound} we know that, with high probability
$1 - \delta$, for at least the 2/3 of the points
$\bar{\vec{w}}$ in $\mathcal{E}$ it is true that
$\bar{\ell}(\vec{w}) - \bar{\ell}(\vec{w}^*) \le \eta$ where
$\eta = \frac{3 C \cdot h(r^*)}{\beta_*^5} \cdot \frac{d^2}{M}
\left( 1 + \log(M) \right)$. Moreover we will prove later that
$\bar{\ell}(\vec{w}) - \bar{\ell}(\vec{w}^*) \le \eta$ implies
$\norm{\vec{w} - \vec{w}^*} \le c \cdot \eta$, where $c$ is a universal
constant. Therefore with high probability $1 - \delta$ for at least the 2/3 of
the points $\bar{\vec{w}}, \bar{\vec{w}'}$ in $\mathcal{E}$ it is true that
$\norm{\vec{w} - \vec{w}'} \le 2 c \cdot \eta$. Hence if we set $\hat{\vec{w}}$
to be a point that is at least $2 c \cdot \eta$ close to more that the half of
the points in $\mathcal{E}$ then with high probability $ 1- \delta$ we have that
$f(\vec{\bar{w}}) - f(\vec{w}^*) \le \eta$.

\noindent Hence we can condition on the event $f(\vec{\hat{w}}) - f(\vec{w}^*)
\le \frac{2 C \cdot h(r^*)}{\beta_*^5} \cdot \frac{d^2}{M} \left( 1 + \log(M)
\right)$ and we only lose probability at most $\delta$. Now remember that for
Lemma \ref{lem:expectedPerformanceSGD} to apply we have
$r^* = O\left( \frac{\log \left(1 / \alpha\right)}{\alpha^2}\right)$. Also,
using Corollary \ref{corollary:box mass} we get that $\beta^* \ge c_{\alpha}$
where is a constant $c_{\alpha}$ that depends only on the constant $\alpha$.
Hence, with probability at least $1 - \delta$ we have that
\[ f(\vec{\hat{w}}) - f(\vec{w}^*) \le c'_{\alpha} \cdot \frac{d^2}{M}
\left( 1 + \log(M) \right), \]
\noindent where $c'_{\alpha}$ is a constant that depends only on $\alpha$. Now
we can use Lemma \ref{lem:strongConvexityToDistance} to get that
\begin{equation} \label{eq:finalProofSecondToLast}
  \norm{\hat{\vec{w}} - \vec{w}^*}_2 \le c''_{\alpha} \sqrt{\frac{d^2}{M} \left(1 + \log (M)\right)}.
\end{equation}
\noindent Also, it holds that
\[ \norm{\hat{\vec{w}} - \vec{w}^*}_2^2 =
     \norm{\vec{\nu} - \vec{\nu}^*}_2^2 +
     \norm{\matr{T} - \matr{T}^*}_F^2 =
     \norm{\matr{\Sigma}^{-1} \vec{\mu} - \matr{\Sigma}^{* -1} \vec{\mu}^*}_2^2
          + \norm{\matr{\Sigma}^{-1} - \matr{\Sigma}^{* -1}}_F^2. \]
Hence, for $M \ge \tilde{O} \left( \frac{d^2}{\eps^2} \right)$ and using
\eqref{eq:finalProofSecondToLast} we have that \[ \norm{\matr{\Sigma}^{-1}
\vec{\mu} - \matr{\Sigma}^{* -1} \vec{\mu}^*}_2^2 + \norm{\matr{\Sigma}^{-1} -
\matr{\Sigma}^{* -1}}_F^2 \le \eps. \] The number of samples is $O(K M)$ and the
running time is $\poly(K, M, 1/\eps, d)$. Hence for $K = \log(1/\delta)$ and $M
\ge \tilde{O} \left( \frac{d^2}{\eps^2} \right)$ our theorem follows.

\clearpage
  \begin{algorithm}[H]
  \caption{Projected Stochastic Gradient Descent. Given access to samples from $\normal(\vec{\mu}^*, \matr{\Sigma}^*, S)$.}
  \label{alg:projectedSGD}
  \begin{algorithmic}[1]
    \Procedure{Sgd}{$M, \lambda$}\Comment{$M$: number of steps, $\lambda$: parameter}
    \State Compute $\hat{\vec{\mu}}_S$ and $\hat{\matr{\Sigma}}_S$ and apply
      affine transformation so that $\hat{\vec{\mu}}_S = \vec{0}$ and
      $\hat{\matr{\Sigma}}_S = \matr{I}$
    \State $\vec{w}^{(0)} \gets \begin{bmatrix} \matr{I}^{\flat} \\ \vec{0} \end{bmatrix}$
    \For{$i = 1, \dots, M$}
      \State Sample $\vec{x}^{(i)}$ from $\samplO$
      \State $\eta_i \gets \frac{1}{\lambda \cdot i}$
      \State $\vec{v}^{(i)} \gets \Call{GradientEstimation}{\vec{x}^{(i)}, \vec{w}^{(i - 1)}}$
      \State $\vec{r}^{(i)} \gets \vec{w}^{(i - 1)} - \eta_i \vec{v}^{(i)}$
      \State $\vec{w}^{(i)} \gets \Call{ProjectToDomain}{\vec{r}^{(i)}}$
    \EndFor
    \State $\bar{\vec{w}} \gets \frac{1}{M} \sum_{i = 1}^M \vec{w}^{(i)}$ \Comment{Output the average.}
    \State $\bar{\vec{w}} \leftarrow \matr{\Sigma}_S^{-1/2} \bar{\vec{w}} + \hat{\vec{\mu}}_S$ \Comment{Apply inverse affine transformation}
    \State \textbf{return} $\bar{\vec{w}}$
  \EndProcedure
  \end{algorithmic}
  \end{algorithm}
  \vspace{-15pt}
  \begin{algorithm}[H]
  \caption{The function to estimate the gradient of log-likelihood as in \eqref{eq:gradientofLikelihoodOneSample}.}
  \label{alg:estimateGradient}
  \begin{algorithmic}[1]
    \Function{GradientEstimation}{$\vec{x}, \vec{w}$}\Comment{$\vec{x}$: sample from $\normal(\vec{\mu}, \matr{\Sigma}, S)$}
    \State $\begin{bmatrix} \matr{T}^{\flat} \\ \vec{\nu} \end{bmatrix} \gets \vec{w}$
    \State $\vec{\mu} \gets \matr{T}^{-1} \vec{\nu}$
    \State $\matr{\Sigma} \gets \matr{T}^{-1}$
    \Repeat
    \State Sample $\vec{y}$ from $\normal(\vec{\mu}, \matr{\Sigma})$
    \Until{$M_S\left(\vec{y}\right) = 1$} \Comment{$M_S$ is the membership oracle of the set $S$.}
    \State \textbf{return} $- \begin{bmatrix} \left( - \frac{1}{2} \vec{x} \vec{x}^T \right)^{\flat} \\ \vec{x} \end{bmatrix}
              + \begin{bmatrix} \left( - \frac{1}{2} \vec{y} \vec{y}^T \right)^{\flat} \\ \vec{y} \end{bmatrix}$
  \EndFunction
  \end{algorithmic}
  \end{algorithm}
  \vspace{-15pt}
  \begin{algorithm}[H]
  \caption{The function that projects a current guess back to the domain $\Domain$.}
  \label{alg:projectToDomain}
  \begin{algorithmic}[1]
    \Function{ProjectToDomain}{$\vec{r}$} \Comment{let $r$ be the parameter of the domain $\Domain_r$}
    \State $\begin{bmatrix} \matr{T}^{\flat} \\ \vec{\nu} \end{bmatrix} \gets \vec{r}$
    \State $\vec{\nu}' \gets \arg \min_{\vec{b} : \norm{\vec{b}} \le r_1} \norm{\vec{b} - \vec{\nu}}_2^2$
    \Repeat ~binary search over $\lambda$
    \State solve the projection problem
           \[ \matr{T}' \gets \arg \min_{\matr{T'} \succeq r_3 \matr{I}} \norm{\matr{T} -  \matr{T}'}_F^2 + \lambda \norm{\matr{I} - \matr{T}'}_F^2 \]
    \Until{a $\matr{T}'$ is found with minimum objective value and $\norm{\matr{I} -
           \matr{T}'}_F^2 \le r_2^2$}
    \State \textbf{return} $(\vec{\nu}', \matr{T}')$
  \EndFunction
  \end{algorithmic}
  \end{algorithm}
  \vspace{-25pt}
  \begin{figure}[H]
    ~
    \caption{Description of the Stochastic Gradient Descent (SGD) algorithm for estimating the parameters of a truncated Normal.}
    \label{fig:algorithm}
  \end{figure}
  \clearpage

\section{Impossibility of estimation with an unknown truncation set} \label{sec:lower bound}

  We have showed that if one assumes query access to the truncation set, the
estimation problem for truncated Normals can be efficiently solved with very
few queries and samples.

  If the truncation set is unknown, as we show in this section, it is
information theoretically impossible to produce an estimate that is closer
than a constant in total variation distance to the true distribution even for
single-dimensional truncated Gaussians.

  To do this, we consider two Gaussian distributions
$\normal(\mu_1, \sigma_1^2)$ and $\normal(\mu_2, \sigma_2^2)$ with
\[ \dtv(\normal(\mu_1, \sigma_1^2), \normal(\mu_2, \sigma_2^2)) = \alpha.\]
We show that there exist a distribution over truncation sets $S_1$ with
$\normal(\mu_1, \sigma_1^2; S_1)$, $\mathcal{D}_1$, and a distribution over
truncation sets $S_2$ with $\normal(\mu_2, \sigma_2^2; S_2)$, $\mathcal{D}_2$,
such that a random instance $\lbrace S_1, \normal(\mu_1, \sigma_1^2) \rbrace$
with $S_1$ drawn from $\mathcal{D}_1$ is indistinguishable from a random
instance $\lbrace S_2, \normal(\mu_2, \sigma_2^2) \rbrace$ with $S_2$ drawn
from $\mathcal{D}_2$.

\begin{lemma}[Indistinguishability with unknown set] \label{lem:indistinguishability}
    Consider two single-dimensional Gaussian distributions
  $\normal(\mu_1, \sigma_1^2)$ and $\normal(\mu_2, \sigma_2^2)$ with
  $\dtv(\normal(\mu_1, \sigma_1^2), \normal(\mu_2, \sigma_2^2)) = 1 - \alpha$.
  The truncated Gaussian
  $\normal(\mu_1, \sigma_1^2, S_1)$ with an unknown set $S_1$ such that
  $\normal(\mu_1, \sigma_1^2; S_1) = \alpha$ is indistinguishable from the
  distribution $\normal(\mu_2, \sigma_2^2; S_2)$ with unknown set $S_2$ such
  that $\normal(\mu_2, \sigma_2^2; S_2) = \alpha$.
\end{lemma}

\begin{proof}
    We define the randomized family of sets $\mathcal{D}_i$:
  The random set $S_i \sim \mathcal{D}_i$ is constructed by adding every point
  $x \in \mathbb{R}$ with probability
  $\min\{\normal(\mu_{3 - i}, \sigma^2_{3 - i}; x)/\normal(\mu_i, \sigma^2_i; x), 1\}$.

    Now consider the distribution $p$ with density
  $\frac{1}{\alpha} \min\{ \normal(\mu_1, \sigma_1^2; x), \normal(\mu_2,   \sigma_2^2; x) \}$.
  This is a proper distribution as
  $\dtv(\normal(\mu_1, \sigma_1^2), \normal(\mu_2, \sigma_2^2)) = 1 - \alpha$.
  Note that samples from $p$ can be generated by performing the following
  rejection sampling process: Pick a sample from the distribution
  $\normal(\mu_i, \sigma^2_i)$ and reject it with probability
  $\min\{ \normal(\mu_{3-i}, \sigma^2_{3 - i}; x)/\normal(\mu_i, \sigma^2_i; x), 1 \}$.

    We now argue that samples from the distribution
  $\normal(\mu_i, \sigma^2_i; S_i)$ for a random $S_i \sim \mathcal{D}_i$ are
  indistinguishable from $p$. This is because an alternative way of sampling
  the distribution $\normal(\mu_i, \sigma^2_i; S_i)$ can be sampled as
  follows. Draw a sample $x$ from  $\normal(\mu_i, \sigma^2_i)$ and then check
  if $x \in S_i$. By the principle of deferred randomness, we may not commit
  to a particular set $S_i$ only decide whether $x \in S_i$ after drawing $x$
  as long as the selection is consistent. That is every time we draw the same
  $x$ we must output the same answer. This sampling process is identical to
  the sampling process of $p$ until the point where an $x_i$ is sampled twice.
  As the distributions are continuous and every time a sample is accepted with
  probability $\alpha > 0$ no collisions will be found in this process.
\end{proof}

  The following corollary completes the proof of Theorem~\ref{thm:lower bound}.

\begin{corollary} \label{cor:lower bound}
    For all $\alpha>0$, given infinitely many samples from a univariate
  normal ${\cal N}(\mu, \sigma^2)$, which are truncated to an unknown set $S$
  of measure $\alpha$, it is impossible to estimate parameters $\hat{\mu}$ and
  $\hat{\sigma}^2$ such that the distributions  ${\cal N}(\mu, \sigma^2)$ and
  ${\cal N}(\hat{\mu}, \hat{\sigma}^2)$ are guaranteed to be within
  $\frac{1-\alpha}{2}$.
\end{corollary}

  To see why the corollary is true, note that since it is impossible to
distinguish between the truncated Gaussian distributions
$\normal(\mu_1, \sigma_1^2, S_1)$ and
$\normal(\mu_2, \sigma_2^2, S_2)$ with
$\dtv(\normal(\mu_1, \sigma_1^2), \normal(\mu_2, \sigma_2^2)) > 1 - \alpha$,
any estimated ${\cal N}(\hat{\mu}, \hat{\sigma}^2)$ will satisfy either
\[ \dtv(\normal(\mu_1, \sigma_1^2), \normal(\hat{\mu}, \hat{\sigma}^2)) > \frac{1 - \alpha} 2 \quad \text{ or } \quad \dtv(\normal(\mu_2, \sigma_2^2), \normal(\hat{\mu}, \hat{\sigma}^2)) > \frac {1 - \alpha} 2. \]

\begin{remark}
    The construction in Lemma~\ref{lem:indistinguishability} uses random sets
  $S_1$ and $S_2$ that select each point on the real line with some
  probability. One may use coarser sets by including all points within some
  range $\eps$ of the randomly chosen points. In this case the collision
  probability is no longer $0$ and depends on $\eps$. Given that no collisions
  are seen the two cases are again indistinguishable. Moreover, for very small
  $\eps$, an extremely large number of samples are needed to see a collision.
\end{remark}

\section*{Acknowledgements}
The authors were supported by NSF CCF-1551875, CCF-1617730, CCF-1650733,
CCF-1733808, CCF-1740751 and IIS-1741137.

  \bibliographystyle{alpha}
  \bibliography{ref}

  \appendix

\section{Missing Proofs from Section \ref{sec:upper bound}} \label{app:upper bound}

\subsection*{Proof of Claim \ref{clm:covarianceWithoutTruncationDominance}} \label{app:proof:clm:covarianceWithoutTruncationDominance}

We first define the following matrices
\[ Q(\vec{\mu}, \matr{\Sigma}) \triangleq \frac{1}{2} \matr{\Sigma} \otimes \matr{\Sigma} + \frac{1}{4} (\vec{\mu} \vec{\mu}^T) \otimes \matr{\Sigma} + \frac{1}{4} \vec{\mu} \otimes \matr{\Sigma} \otimes \vec{\mu}^T + \frac{1}{4} \vec{\mu}^T \otimes \matr{\Sigma} \otimes \vec{\mu} + \frac{1}{4} \matr{\Sigma} \otimes (\vec{\mu} \vec{\mu}^T) \]
\[ R(\vec{\mu}, \matr{\Sigma}) \triangleq \frac{1}{2} \left( \vec{\mu} \otimes \matr{\Sigma} + \matr{\Sigma} \otimes \vec{\mu} \right) \]
\[ B(\vec{\mu}, \matr{\Sigma}) \triangleq \frac{1 + \rho}{4} (\vec{\mu} \vec{\mu}^T) \otimes \matr{\Sigma} + \frac{1 + \rho}{4} \vec{\mu} \otimes \matr{\Sigma} \otimes \vec{\mu}^T + \frac{1 + \rho}{4} \vec{\mu}^T \otimes \matr{\Sigma} \otimes \vec{\mu} + \frac{1 + \rho}{4} \matr{\Sigma} \otimes (\vec{\mu} \vec{\mu}^T) \]
and we can prove that the following holds.
\begin{align} \label{clm:covarianceWithoutTruncation}
  \Cov_{\vec{z} \sim \normal(\vec{\mu},
                 \matr{\Sigma})} \left[
                 \begin{bmatrix}
                   \left( - \frac{1}{2} \vec{z}
                    \vec{z}^T \right)^{\flat} \\
                   \vec{z}
                 \end{bmatrix},
                 \begin{bmatrix}
                   \left( - \frac{1}{2} \vec{z}
                     \vec{z}^T \right)^{\flat} \\
                   \vec{z}
                 \end{bmatrix} \right] =
                 \begin{bmatrix}
                   Q(\vec{\mu}, \matr{\Sigma}) & R(\vec{\mu}, \matr{\Sigma}) \\
                   R^T(\vec{\mu}, \matr{\Sigma}) & \matr{\Sigma}
                 \end{bmatrix}.
\end{align}
Once we have proved the above equation we have that
\begin{align} \label{clm:covarianceWithoutTruncationDomination:proof:1}
  \begin{bmatrix}
    Q(\vec{\mu}, \matr{\Sigma}) & R(\vec{\mu}, \matr{\Sigma}) \\
    R^T(\vec{\mu}, \matr{\Sigma}) & \matr{\Sigma}
  \end{bmatrix}
  -
  \begin{bmatrix}
    \tilde{Q}(\vec{\mu}, \matr{\Sigma}) & \matr{0} \\
    \matr{0} & \frac{\rho - 3}{1 + \rho} \cdot \matr{\Sigma}
  \end{bmatrix}
  =
  \begin{bmatrix}
    B(\vec{\mu}, \matr{\Sigma}) & - R(\vec{\mu}, \matr{\Sigma}) \\
    - R^T(\vec{\mu}, \matr{\Sigma}) & \frac{4}{1 + \rho} \cdot \matr{\Sigma}
  \end{bmatrix}
  \triangleq \matr{D}
\end{align}
and our goal is to show that the matrix $\matr{D}$ on the right hand side is
positive semi-definite. Using Schur's complement theorem for block symmetric
matrices we have that $\matr{D} \succeq \matr{0}$ is implied by the following
conditions
\[ \frac{4}{1 + \rho} \cdot \matr{\Sigma} \succ \matr{0} \quad \mathrm{ and } \quad B(\vec{\mu}, \matr{\Sigma}) - \frac{1 + \rho}{4} \cdot R^T(\vec{\mu}, \matr{\Sigma}) \cdot \matr{\Sigma}^{-1} \cdot R(\vec{\mu}, \matr{\Sigma}) \succeq \matr{0}. \]
\noindent The first condition is implied by our assumption that we consider only
covariance matrices that have full rank, as we will also see later in the proof
of our algorithm. On the other hand, if we use the mixed product property of the
tensor product together with the easy to check identities
$\vec{\mu}^T \otimes \matr{\Sigma} \otimes \vec{\mu} = (\matr{I} \otimes \vec{\mu}) \cdot (\vec{\mu}^T \otimes \matr{\Sigma})$ and
$\vec{\mu} \otimes \matr{\Sigma} \otimes \vec{\mu}^T = (\vec{\mu} \otimes \matr{I}) \cdot (\matr{\Sigma} \otimes \vec{\mu}^T)$ we can do algebraic
calculations and prove that
\[ B(\vec{\mu}, \matr{\Sigma}) - \frac{1 + \rho}{4} \cdot R^T(\vec{\mu}, \matr{\Sigma}) \cdot \matr{\Sigma}^{-1} \cdot R(\vec{\mu}, \matr{\Sigma}) = \matr{0}. \]
\noindent Therefore we conclude that $\matr{D} \succeq \matr{0}$ and the claim
follows. To conclude we need to prove \eqref{clm:covarianceWithoutTruncation}.

\begin{proof}[Proof of Equation \eqref{clm:covarianceWithoutTruncation}]
    We first observe that
  $\left(\vec{z} \vec{z}^T\right)^{\flat} = \vec{z} \otimes \vec{z}$. Then we
  can write $\vec{z} = \vec{y} + \vec{\mu}$ where $\vec{y}$ follows the
  distribution $\normal(\vec{0}, \matr{\Sigma})$. If we substitute $\vec{z}$
  with $\vec{y} + \vec{\mu}$ then we can use Isserlis' Theorem, also known as
  Wick's probability theorem, \cite{Isserlis1918, Wick1950} together with
  Theorem 4.12 from \cite{DKK+16b} and the equation
  \eqref{clm:covarianceWithoutTruncation} follows.
\end{proof}

\subsection*{Proof of Claim \ref{clm:covarianceWithoutTruncationEigenvalues}} \label{app:proof:clm:covarianceWithoutTruncationEigenvalues}

    The fact that
  $\frac{\rho - 3}{1 + \rho} \cdot \matr{\Sigma} \succeq \frac{\sigma_m(\matr{\Sigma})}{16 \norm{\vec{\mu}}_2^2 + \sqrt{\sigma_m(\matr{\Sigma})}} \cdot \matr{I}_d$ follows directly from the
  definition of $\rho$ and $\sigma_m(\matr{\Sigma})$ so it remains to prove that
  $\tilde{Q}(\vec{\mu}, \matr{\Sigma}) \succeq \frac{\sigma_m(\matr{\Sigma})}{4} \cdot \matr{I}_{d^2}$. We remind that
  \[ \tilde{Q}(\vec{\mu}, \matr{\Sigma}) = \frac{1}{2} \matr{\Sigma} \otimes \matr{\Sigma} - \frac{\rho}{4} \cdot (\vec{\mu} \vec{\mu}^T) \otimes \matr{\Sigma} - \frac{\rho}{4} \cdot \vec{\mu} \otimes \matr{\Sigma} \otimes \vec{\mu}^T - \frac{\rho}{4} \cdot \vec{\mu}^T \otimes \matr{\Sigma} \otimes \vec{\mu} - \frac{\rho}{4} \cdot \matr{\Sigma} \otimes (\vec{\mu} \vec{\mu}^T). \]
  We first observe that
  $\frac{1}{\norm{\vec{\mu}}_2^2} \cdot (\vec{\mu} \vec{\mu}^T) \otimes \matr{\Sigma} = (\vec{w} \vec{w}^T) \otimes \matr{\Sigma}$,
  where $\vec{w}$ is a unit vector and hence
  \begin{align} \label{eq:proof:clm:covarianceWithoutTruncationEigenvalues:1}
    \frac{1}{8} \matr{\Sigma} \otimes \matr{\Sigma} - \frac{\rho}{4} \cdot (\vec{\mu} \vec{\mu}^T) \otimes \matr{\Sigma} & = \left( \frac{1}{8} \cdot \matr{\Sigma} - \frac{\sqrt{\sigma_m(\matr{\Sigma})}}{16} \cdot (\vec{w} \vec{w}^T) \right) \otimes \matr{\Sigma} \succeq \frac{\sigma_m(\matr{\Sigma})}{16} \cdot \matr{I}
  \end{align}
  where the last inequality follows from the fact that the eigenvalues of the
  tensor product are the products of the eigenvalues and the fact that
  $2 \matr{\Sigma} - \sqrt{\sigma_m(\matr{\Sigma})} \cdot (\vec{w} \vec{w}^T) \succeq \sqrt{\sigma_m(\matr{\Sigma})} \cdot \matr{I}$.
  The latter follows from the definition of $\sigma_m(\matr{\Sigma})$ and the
  fact that $\vec{w} \vec{w}^T \preceq \matr{I}$ since $\vec{w}$ is a unit
  vector.
  \smallskip

    Next we consider the matrix
  $2 \matr{\Sigma} \otimes \matr{\Sigma} - 4 \rho \cdot \vec{\mu}^T \otimes \matr{\Sigma} \otimes \vec{\mu}$
  which is easy to see that it is equal to the matrix
  $\matr{E} \triangleq 2 \cdot \matr{\Sigma} \otimes \matr{\Sigma} - \sqrt{\sigma_m(\matr{\Sigma})} \cdot (\matr{I} \otimes \vec{w}) \cdot (\vec{w}^T \otimes \matr{\Sigma})$,
  where $\vec{w}$ is a unit vector. Now if we prove that
  $\matr{\Sigma}^{-1} \cdot \matr{E} \cdot \matr{\Sigma}^{-1} \succeq \matr{I}$
  then we can conclude that
  $\matr{E} \succeq \sigma_m(\matr{\Sigma}) \cdot \matr{I}$ since the maximum
  eigenvalue of $\matr{\Sigma}^{- 1}$ is $1/\sqrt{\sigma_m(\matr{\Sigma})}$.
  Therefore we consider the matrix
  $\matr{F} \triangleq \matr{\Sigma}^{-1} \cdot \matr{E} \cdot \matr{\Sigma}^{-1}$
  which using the mixed product property of the tensor product becomes
  \[ \matr{F} = 2 \matr{I}_{d^2} - \sqrt{\sigma_m(\matr{\Sigma})} \cdot (\matr{\Sigma}^{-1} \otimes \vec{w}) \cdot (\vec{w}^T \otimes \matr{I}_d) \]
  and therefore for any unit vector $\vec{u}$ in $\reals^{d^2}$ we have that
  \begin{align} \label{eq:proof:clm:covarianceWithoutTruncationEigenvalues:2:a}
    \vec{u}^T \cdot \matr{F} \cdot \vec{u} = 2 -  \sqrt{\sigma_m(\matr{\Sigma})} \cdot \vec{u}^T \cdot (\matr{\Sigma}^{-1} \otimes \vec{w}) \cdot (\vec{w}^T \otimes \matr{I}_d) \cdot \vec{u}.
  \end{align}
  Now observe that
  \[ \norm{\vec{u}^T \cdot (\matr{\Sigma}^{-1} \otimes \vec{w})}_2^2 = \vec{u}^T \cdot (\matr{\Sigma}^{-1} \otimes \vec{w}) \cdot (\matr{\Sigma}^{-1} \otimes \vec{w}^T) \cdot \vec{u} = \vec{u}^T \cdot (\matr{\Sigma}^{-2} \otimes \vec{w} \vec{w}^T) \cdot \vec{u} \]
  but since $\vec{w}$ is a unit vector we have that the maximum eigenvalue of
  $(\matr{\Sigma}^{-2} \otimes \vec{w} \vec{w}^T)$ is
  $1/\sigma_m(\matr{\Sigma})$ and hence we get that
  \begin{align} \label{eq:proof:clm:covarianceWithoutTruncationEigenvalues:2:b}
    \norm{\vec{u}^T \cdot (\matr{\Sigma}^{-1} \otimes \vec{w})}_2 \le \frac{1}{\sqrt{\sigma_m(\matr{\Sigma})}}.
  \end{align}
  Also if we use again the fact that $\vec{w}$ is a unit vector we have that
  \begin{align} \label{eq:proof:clm:covarianceWithoutTruncationEigenvalues:2:c}
    \norm{(\vec{w}^T \otimes \matr{I}_d) \cdot \vec{u}}_2^2 =  \vec{u}^T \cdot (\vec{w} \otimes \matr{I}_d) \cdot (\vec{w}^T \otimes \matr{I}_d) \cdot \vec{u} \le 1.
  \end{align}
  If we put together
  \eqref{eq:proof:clm:covarianceWithoutTruncationEigenvalues:2:a},
  \eqref{eq:proof:clm:covarianceWithoutTruncationEigenvalues:2:b}, and
  \eqref{eq:proof:clm:covarianceWithoutTruncationEigenvalues:2:c} we get that
  \[ \vec{u}^T \cdot \matr{F} \cdot \vec{u} \ge 1 \]
  and therefore $\matr{F} \succeq \matr{I}$ which implies
  $\matr{E} \succeq \sigma_m(\matr{\Sigma}) \cdot \matr{I}$ and hence we have
  that
  \begin{align} \label{eq:proof:clm:covarianceWithoutTruncationEigenvalues:2}
    \frac{1}{8} \matr{\Sigma} \otimes \matr{\Sigma} - \frac{\rho}{4} \cdot \vec{mu}^T \otimes \matr{\Sigma} \otimes \vec{\mu} \succeq \frac{\sigma_m(\matr{\Sigma})}{16} \cdot \matr{I}.
  \end{align}
  Similarly to \eqref{eq:proof:clm:covarianceWithoutTruncationEigenvalues:1}
  and \eqref{eq:proof:clm:covarianceWithoutTruncationEigenvalues:2} we can prove
  that
  \[ \frac{1}{8} \matr{\Sigma} \otimes \matr{\Sigma} - \frac{\rho}{4} \cdot \matr{\Sigma} \otimes (\vec{\mu} \vec{\mu}^T) \succeq \frac{\sigma_m(\matr{\Sigma})}{16} \cdot \matr{I}, \text{ and that} \]
  \[ \frac{1}{8} \matr{\Sigma} \otimes \matr{\Sigma} - \frac{\rho}{4} \cdot \vec{mu} \otimes \matr{\Sigma} \otimes \vec{\mu}^T \succeq \frac{\sigma_m(\matr{\Sigma})}{16} \cdot \matr{I}. \]
  These combined with
  \eqref{eq:proof:clm:covarianceWithoutTruncationEigenvalues:1} and
  \eqref{eq:proof:clm:covarianceWithoutTruncationEigenvalues:2} imply that
  $\tilde{Q}(\vec{\mu}, \matr{\Sigma}) \succeq \frac{\sigma_m(\matr{\Sigma})}{4} \cdot \matr{I}_{d^2}$
  and the claim follows.

\subsection*{Proof of Claim \ref{clm:eigenvaluesofRprimevsRstar}} \label{app:proof:clm:eigenvaluesofRprimevsRstar}

    We have that the matrices $\matr{R}'$ and $\matr{R}^*$ are of the form
  \[ \matr{R}' = \Exp_{\vec{w} \sim \mathcal{D}} \left[ (\vec{w} - \vec{q}) (\vec{w} - \vec{q})^T \right] \]
  \[ \matr{R}^* = \Exp_{\vec{w} \sim \mathcal{D}} \left[ (\vec{w} - \Exp_{\vec{w} \sim \mathcal{D}}[\vec{w}]) (\vec{w} - \Exp_{\vec{w} \sim \mathcal{D}}[\vec{w}])^T \right] \]
  \noindent where $\vec{w}$ corresponds to the vector
  $\begin{bmatrix}
    \left( - \frac{1}{2} \vec{z}
     \vec{z}^T \right)^{\flat} \\
    \vec{z}
  \end{bmatrix}$, $\mathcal{D}$ is the distribution of $\vec{w}$ when $\vec{z}$
  follows the distribution $\normal(\vec{\mu}, \matr{\Sigma})$ and $\vec{q}$ is
  the expected value of $\vec{w}$ when $\vec{z}$ follows the truncated normal
  distribution $\normal(\vec{\mu}, \matr{\Sigma}, S)$. In order to prove Claim
  \ref{clm:eigenvaluesofRprimevsRstar} we prove the more general statement that
  $\matr{R}' \succeq \matr{R}^*$ for any distribution $\mathcal{D}$ and any
  vector $\vec{q}$. Let $\vec{u}$ be an arbitrary real vector with unit length
  and dimension $d^2 + d$. We have that
  \[ \vec{u}^T \matr{R}' \vec{u} = \Exp_{\vec{w} \sim \mathcal{D}}\left[ (\vec{u}^T \vec{w} - \vec{u}^T \vec{q})^2 \right], \text{ and } \]
  \[ \vec{u}^T \matr{R}^* \vec{u} = \Exp_{\vec{w} \sim \mathcal{D}} \left[ (\vec{u}^T \vec{w} - \Exp_{\vec{w} \sim \mathcal{D}}[\vec{u}^T \vec{w}])^2 \right]. \]
  \noindent Hence if we define $\mathcal{D}_{\vec{u}}$ to be the distribution of
  $\vec{u}^T \vec{w}$ and $t \triangleq \vec{u}^T \vec{q}$ we have that
  \[ \vec{u}^T \matr{R}' \vec{u} = \Exp_{w \sim \mathcal{D}_{\vec{u}}}\left[ (w - t)^2 \right], \text{ and } \]
  \[ \vec{u}^T \matr{R}^* \vec{u} = \Exp_{w \sim \mathcal{D}_{\vec{u}}} \left[ (w - \Exp_{w \sim \mathcal{D}_{\vec{u}}}[w])^2 \right]. \]
  Now it is very easy to see that
  $\Exp_{w \sim \mathcal{D}_{\vec{u}}}[w] = \argmin_{x \in \reals} \Exp_{w \sim \mathcal{D}_{\vec{u}}}[(w - x)^2]$
  from which we can immediately see that
  \[ \Exp_{w \sim \mathcal{D}_{\vec{u}}}\left[ (w - t)^2 \right] \ge \Exp_{w \sim \mathcal{D}_{\vec{u}}} \left[ (w - \Exp_{w \sim \mathcal{D}_{\vec{u}}}[w])^2 \right] \]
  and hence $\vec{u}^T \matr{R}' \vec{u} \ge \vec{u}^T \matr{R}^* \vec{u}$.
  Since we picked $\vec{u}$ arbitrarily, the above holds for any unit vector
  $\vec{u}$ and therefore the claim follows.
 \end{document}